\documentclass[11pt]{article}
\usepackage{mathrsfs}
\usepackage{amsmath}
\usepackage{amsfonts}
\usepackage{amssymb}
\usepackage{epsfig}
\usepackage{dsfont}
\usepackage{multirow}
\usepackage{subfigure}
\usepackage{parskip}
\usepackage{tikz}
\usepackage{color}
\usepackage{extarrows}
\usepackage{booktabs}
\usepackage{enumerate}
\usepackage[figuresright]{rotating}
\usepackage{amsthm}
\usepackage[numbers,sort]{natbib}
\usepackage[noend]{algpseudocode}
\usepackage{algorithmicx,algorithm}
\usepackage{indentfirst}

\usepackage{algpseudocode}
\usepackage{multirow}
\usepackage{appendix}

\newtheorem{thm}{Theorem}[section]
\newtheorem{lema}{Lemma}[section]

\newtheorem{assu}{Assumption}[section]

\newcommand{\setd}{{ d \kern -.15em l}}
\newcommand{\hatsetd}{ d \hat{\kern -.15em l }}

\newcommand{\bgeqn}{\begin{eqnarray}}
\newcommand{\edeqn}{\end{eqnarray}}
\newcommand{\bgeq}{\begin{eqnarray*}}
\newcommand{\edeq}{\end{eqnarray*}}
\newcommand{\bec}{\begin{center}}
\newcommand{\enc}{\end{center}}
\newcommand{\R}{{\rm I\!R}}

\newcommand{\inmat}[1]{\mbox{\rm {#1}}}

\newcommand{\dd}{\mathsf {d\kern -0.07em l}}

\newcommand{\C}{{\cal C}}

\newcommand{\be}{\begin{equation}}
\newcommand{\ee}{\end{equation}}

\newcommand{\normmm}[1]{{\left\vert\kern-0.25ex\left\vert\kern-0.25ex\left\vert #1
		\right\vert\kern-0.25ex\right\vert\kern-0.25ex\right\vert}}

\def\bbe{{\Bbb{E}}} %expectation

\newcommand{\define}{:=}%\newcommand{\define}{\stackrel{\triangle}{=}}

\begin{document}

\begin{center}

{\Large
	Stochastic Approximation Based Confidence Regions for Stochastic Variational Inequalities }\footnote
{The work is supported by National Natural Science Foundation of China \#11971090.}

\vspace{0.2cm}

Wuwenqing Yan and Yongchao Liu
\\
School of Mathematical Sciences,
Dalian University of Technology, Dalian, 116024, China
({\tt ywwq@mail.dlut.edu.cn  (Yan),  lyc@dlut.edu.cn (Liu)}).
%\today

\end{center}

\noindent{\bf Abstract.}
 The sample average approximation (SAA) and  the stochastic approximation (SA) are two   popular schemes for solving the stochastic variational inequalities problem (SVIP). In the past decades, theories on the  consistency of
 the SAA solutions  and SA solutions  have been {well studied}.  More recently, the asymptotic confidence regions of the true solution to SVIP have {been constructed} when the SAA scheme is implemented.
It is of fundamental interest to
develop confidence regions of the true solution to the SVIP when the {SA scheme} is employed.
In this paper, we discuss the framework of constructing asymptotic    confidence regions for the true solution of SVIP with a focus on stochastic dual average method.
We first establish the asymptotic normality of the SA solutions  both in ergodic sense and non-ergodic sense. Then the online  methods  of estimating the  covariance  matrices in the normal distributions  are studied.
 Finally,   practical procedures of building the asymptotic
confidence regions of solutions to SVIP with
numerical simulations  are presented.


\section{Introduction}

For the given convex set $\C \subset\R^n$ and a mapping $f: \C \rightarrow \R^{n}$,
the variational inequalities problem (VIP) is to find a vector $x \in \C$ such that
$$
(y-x)^Tf(x)\geq 0,\quad \forall y\in \C.
$$
{VIP has many applications in engineering, economics, game theory and has been well studied in
	theories, algorithms, see the monograph by Facchinei and Pang \cite{VI-book03}.}
% in engineering, economics, game
%theory  and has been well studied in theories, algorithms and applications
In order to describe decision
making problems which involve future uncertainty, the stochastic version of variational inequalities problem (SVIP) has been proposed.  Different approaches to   incorporate the uncertainty into VIP induce different SVIP models, such as,
 expected residual minimization-SVIP (ERM-SVIP)
  model \cite{erm-Chen-Fuku},   expected value-SVIP  (EV-SVIP) \cite{robinson-vi}, $\mathcal{L}^p$-SVIP model \cite{Gwinner2000}, two-stage SVIP model \cite{chen17} and multi-stage SVIP model \cite{RW17}.

In this paper, we focus on the EV-SVIP model (for simplicity, we {refer} the EV-SVIP as SVIP):
 find $x\in \C$ such that
\bgeqn
\label{svi}
(y-x)^Tf(x)\geq 0,\quad \forall y\in \C,
\edeqn
where $f(x):=\bbe_P[F(x,\xi)]$, {$\C \subset\R^n$ is a convex set,} $\xi$  is  a random vector    defined on   probability space $(\Omega, \mathscr{F}, P)$   with support set $\Xi$, $F(\cdot,\cdot)$ is measurable function from $\C\times \Xi$ to $\R^n$ and
 $\bbe_P[\cdot]$ denotes the expected value with respect to
the distribution $P$.
 Indeed,  (\ref{svi})  is  deterministic VIP  if  $\bbe_P[F(x,\xi)]$  has a closed form representation.
However, in most problems of interest obtaining a closed form of $\bbe_P[\cdot]$ or computing its value
  numerically is usually difficult either due to the unavailability of  distribution
of $\xi$ or multiple integration involved.
In general, it is more realistic to obtain
a sample of the random vector $\xi $ either from past data or from computer simulation.
 Depending on
how sampling is incorporated with the algorithm, solution methods  for SVIP can be classified into two basic
categories:  sample average approximation (SAA) based and stochastic approximation (SA) based.

SAA  method  is also known under
different names such as Monte Carlo method, sample path optimization, and has been well studied in stochastic programming.
{Suppose there is independent and identically distributed (iid)  sample}
$\xi_1,\cdots, \xi_N$, SAA method replaces the $f(\cdot)$ in (\ref{svi}) with
\begin{equation*}
\label{eq:fn} f_N(\cdot):=\frac{1}{N}\sum_{j=1}^N F(\cdot,\xi_j).
\end{equation*}
Then algorithms for VIP are employed to solve (\ref{svi}) and return the SAA solutions.
Since SAA method does not depend on the algorithms, it  is an `exterior' approach.   {SAA method is} known to be consistent \cite{robinson-vi}, that is, the SAA solutions converge to the true counterpart
with probability one.
{A natural question to ask is how well the SAA solutions approximate the true solution.}
Very recently, {Lu et al.} \cite{LuLASSO,LuMOR13, LuMP18,LuOPT12,LuSIAM14}  study the confidence regions of true solutions to SVIP based on SAA solutions, where the {\em normal map approach} is proposed.
The idea of the normal map approach is to build the confidence region of solution to $F^{\inmat{nor}}_\C(z)=0$,\footnote{The normal map induced by function $f(\cdot)$ and convex set $\C$ reads as:
 $$
 F^{\inmat{nor}}_\C(z) := f(\Pi_\C(z))+z-\Pi_\C(z).
$$} therefore, the confidence region of solution to SVIP (\ref{svi}) {can be obtained through  the relations between the solutions to SVIP (\ref{svi}) and  $F^{\inmat{nor}}_\C(z) = 0$}.
{See \cite{LuLASSO,yu2019confidence} for the application of normal map approach on least absolute shrinkage and selection operator (lasso) and sparse penalized regression.}
  Motivated by the normal map approach,
  Liu et al. \cite{LiuZ19,Liu20} propose the so-called {\em error bound approach} to build the confidence regions of SVIP  by the SAA solutions. The road-map of error bound approach is that  characterizing  the  distance between the SAA solutions and the true solution by error bound conditions first, then statistical tools such as central limit theorem and Owen's empirical likelihood theorem  are used {to build the confidence regions.}

On the other hand, the {SA scheme}   always  depends on the {structure} of the algorithm,  then  it  is an `interior' approach.
The development of stochastic approximation scheme goes back to the work {of} Robbins and Monro \cite{RMmethod}, where the stochastic
root-finding problems are studied. Research on asymptotic normality results for the SA based algorithm can be traced to the works in the 1950s \cite{Chung54,fabian1968asymptotic}. In particular, Polyak and Juditsky \cite{polyak1992acceleration} show that the averaged SA iterates is asymptotically normal with optimal covariance matrix for strongly convex stochastic optimization problem. In \cite{Hsieh2002}, Hsieh and Glynn establish the asymptotically normality of Robbins-Monro algorithm \cite{RMmethod} and construct confidence regions of true solutions through simulating multiple independent replications of the stochastic approximation procedure. More recently, Lei and Shanbhag \cite{lei2020variance} provide a unified frame work to show the  asymptotically normality of variance-reduced accelerated stochastic first-order methods, where the confidence regions of the true solutions are constructed through simulation method \cite{Hsieh2002}.
%Hsieh and Glynn \cite{Hsieh2002} propose the  central limit theorem for the Robbins-Monro algorithm and simulate multiple independent replications of the stochastic approximation procedure to construct confidence regions for SA of Robbins-Monro type. Recent efforts \cite{lei2020variance} have  developed rigorous confidence statements for variance-reduced accelerated schemes.
The first SA based method for SVIP is proposed by Jiang and Xu \cite{Jiang08}. 
Since {it is easy to implement and
needs} less memory, {researches} on SA based methods for SVIP have
been well developed, {for examples,}  SA based  extragradient method \cite{Iusem19},  SA based incremental constraint projection methods \cite{Wang15}, SA based  backward-forward methods \cite{Shanbhag16} and SA based mirror-proximal algorithm \cite{Yousefian18}.  As far as we known, all the results on the SA based {methods} for SVIP focus on the consistency, that is,  under some moderate conditions,  the SA solutions converge to the true counterpart.
 {It is of fundamental interest to use SA  based solutions  to develop
confidence regions of prescribed level of significance for the true solution.}

In this paper, we discuss the framework of constructing asymptotic  confidence {regions of the} true solution to SVIP (\ref{svi})  when stochastic approximation based method is implemented.  The two seminal papers on stochastic approximation  \cite{Duchi19, Chen16} motivate and guide much of our work.
Similar to the normal map approach \cite{LuMOR13,LuMP16,LuSIAM14}, we need to establish the asymptotic normality of SA solutions first. 
Indeed,
% The method is  based on the two new developments on the asymptotic normality of SA method \cite{Duchi19} and the online methods of estimating matrix covariance \cite{Chen16}.
 Duchi and  Ruan \cite{Duchi19}  have established the  asymptotic normality of Polyak-Ruppert averaged  iterates  of
 a variant {of} stochastic dual average algorithm  (SDA) \cite{Nesterov} for solving constrained optimization {problems.
% They also claim that stochastic projected gradient descent methods do not enjoy these properties.
  This} {motivates} us to employ SDA to solve the SVIP (\ref{svi}) and  study   the asymptotic normality of  averaged SA solutions (Theorem \ref{thm:SA-c}). {On the other hand, compared with the last iterate of SDA,  the average of iterates may deviate from  the  solution if   the  initial point  of SDA is far away from the solution and the iteration $k$ is not large enough.}  Then we  also establish the asymptotic normality of  the  last iterate of SDA for SVIP (\ref{svi})  (Theorem \ref{thm:non-ave}).

 With the asymptotic normality of SDA solutions, the following task is to estimate the corresponding covariance matrices.   The standard covariance matrix estimator employs the sample average approximation, where the   history data of SDA is needed.  This requirement loses the advantage of stochastic approximation scheme in terms of
data storage. More recently, the seminal work  \cite{Chen16} provides two online methods plug-in and batch-means to estimate the covariance  matrix  when  vanilla SGD  method is implemented on {unconstrained} stochastic optimization problems.
They show the consistency of the both methods with the convergence rate  $O(k^{-\frac{1}{2}})$   for plug-in method  and  $O(k^{-\frac{1}{8}})$ for batch-means method in expectation sense, where $k$ is the number of iterates.
We extend the plug-in and batch-means methods to stochastic dual {averaging} algorithm for {SVIP (\ref{svi}).
%which is constraint optimization problem.
Due to the existence of constraints}, we only obtain the {almost sure} convergence of the plug-in estimator  and  convergence in distribution of batch-means {estimator.
% without the rate of convergence.
Specifically},  Theorems \ref{T3.1} and \ref{T3.2} present the {almost sure} convergence of plug-in estimators for the covariance
matrices in ergodic and non-ergodic asymptotic normality respectively.    Theorem \ref{XX} shows  that  batch-means estimator  of  covariance matrix in ergodic asymptotic normality is {convergent} in {distribution.
%Moreover, we do not know how to estimate the covariance matrix in non-ergodic asymptotic normality through batch-means method yet.
These} results enable us to build confidence regions of the true solution through the iterates of SDA.

The rest of paper is organized as follows.   Section 2  establishes the asymptotic distribution results of SDA  in ergodic and non-ergodic senses.
  Section 3 discusses the  plug-in method and batch-means method for estimating the corresponding covariance {matrices}.
  {Finally, practical procedures of building the asymptotic confidence regions of solutions to SVIP
  with numerical simulations are presented in Section 4.}

Throughout the paper, 
{$[a]_+$ is the largest integer less than or equal to $a$.}
{$\mathbf{I}_n\in\R^{n\times n}$ denotes the identity matrix, $A^{\dagger}$} denotes Moore-Penrose inverse of matrix $A\in \R^{m\times n}$ {and $\operatorname{tr}(A)$ denotes the trace of a square matrix $A$.}
%For a symmetric matrix $A\in \R^{n\times n}$,  $\lambda_{\max}(A)$ and $\lambda_{\min}(A)$ denote the maximum  and minimum eigenvalue of $A$.
$0_n\in\R^n$ is the vector of all 0s.
%For two infinitesimals $a_k$ and $b_k$, $a_k=o(b_k)$ if ${a_k}/{b_k}\to0$.
% Denote $K$ is total number of iteration.
For any sequences $\left\{a_{k}\right\}$ and $\left\{b_{k}\right\}$ of positive numbers, we write  $a_k=o(b_k)$ if ${a_k}/{b_k}\to0$, $a_{k} \gtrsim b_{k}$ if $a_{k} \geq c b_{k}$ holds for all $k$ large enough and some constant $c>0, a_{k} \lesssim b_{k}$ if $b_{k} \gtrsim a_{k}$ holds, and $a_{k} \asymp b_{k}$ if $a_{k} \gtrsim b_{k}$ and $a_{k} \lesssim b_{k}$.
We denote $a_{k} \lesssim_{r} b_{k}$ if $a_{k} \leq c(\xi) b_{k}$ holds for all $k$ large enough and some positive random variable $c(\xi)<\infty$ almost surely.
For a sequence of random vectors $\left\{\xi_{k}\right\}$ and a random vector $\xi$,  $\xi_{k} \stackrel{d}{\rightarrow} \xi$ denotes the convergence in distribution and $\operatorname{Cov}\left(\xi\right)$ denotes the covariance matrix of random {vector} $\xi$. `a.s.' is short for almost surely.

\section{Asymptotic normality}

 Asymptotic normality plays a significant role in stochastic approximation and its history  can be traced to 1950s \cite{Chung54,fabian1968asymptotic}.  In this section, we study the  asymptotic normality of iterates when SDA is implemented on SVIP (\ref{svi}).
   The dual averaging algorithm  is proposed by  Nesterov \cite{Nesterov} {and further
studied} by many authors \cite{Duchi19,Lee12,Lin10,S_Zhao}. We focus  on the  stochastic variant of dual averaging algorithm proposed in \cite{Duchi19}, which for  SVIP (\ref{svi}) reads as following.

%\begin{algorithm}[h]
%	\caption{Dual averaging algorithm for SVIP (\ref{svi})}
%	\label{AL1}
%\begin{algorithmic}[1]
%		\State{Choose $x_0 \in \R^{n}$, set $z_0=0$ and  $k : = 1$.}
%		\State{Set
%		\begin{equation}
%		\label{recursion}
%		x_{k}=\inmat{argmin}_{x\in\C}\left\{\left\langle z_{k-1}, x\right\rangle+\frac{1}{2}\|x\|_{2}^{2}\right\}.
%		\end{equation}}
%		\State{Choose independent and identically distributed  sample $\xi_{k}$ with respect to probability $P$ and compute $F(x_k, \xi_k)$. Set $z_{k}=z_{k-1}+\alpha_{k}F(x_k, \xi_k)$.}
%		
%		\State{If the termination criterion does not meet, set $k=k+1$ and go to Step 2.}
%\end{algorithmic}
%\end{algorithm}

{\begin{algorithm}[h]
		\caption{Stochastic dual averaging algorithm for SVIP (\ref{svi})}
			\label{AL1}
			\hspace*{0.02in} {\bf Input} {$x_0=0_n$}, $z_0=0_n$ {and step-size} $\{\alpha_{k}\}$.%算法的输入， \hspace*{0.02in}用来控制位置，同时利用 \\ 进行换行
		\begin{algorithmic}[1]
		\For{$k=1,2,\cdots$}
			\State{Update
				\begin{equation}
				\label{recursion}
				x_{k}=\inmat{argmin}_{x\in\C}\left\{\left\langle z_{k-1}, x\right\rangle+\frac{1}{2}\|x\|_{2}^{2}\right\}.
				\end{equation}}
			\State {Generate iid  sample $\xi_{k}$ and calculate $F(x_k, \xi_k).$}
			\State Update $z_{k}=z_{k-1}+\alpha_{k}F(x_k, \xi_k).$
%			\State{If the termination criterion does not meet, set $k=k+1$ and go to Step 2.}
		\EndFor
		\State   {\bf end}
		\end{algorithmic}
%		\hspace*{0.02in} {\bf Return:} $x_k$
	\end{algorithm}}

In what follows,  we focus on  the  case that  the set $\C$ in SVIP (\ref{svi}) is {polyhedral}, that is,
%
%For studying the asymptotic normality of Algorithm  \ref{AL1}, we consider the constraint set $\C$ in SVIP (\ref{svi}) is defined by linear inequalities,
$$
\C=\left\{x \in \R^{n}: A x-b \leq 0, \quad D x-d \leq 0\right\},
$$
where $A \in \R^{m_{1} \times n}, b \in \R^{m_{1}}, D\in \R^{m_{2} \times n}$ and $d \in \R^{m_{2}}$.
Let  $x^*\in\C$ be {a} solution to SVIP (\ref{svi}). Without loss of generality, {we assume $A x^*-b=0$   and $D x^*-d<0$, that is, $Ax-b\leq 0$ is the active constraints at the solution $x^*$.}

We next  record the assumptions that   will be used to analyze the asymptotic normality of SDA, which are variations of the {standard conditions on} optimization problem in \cite{Duchi19}.

\begin{assu}
	\label{assu_1}
		Let $x^*\in\C$ be the unique solution to SVIP (\ref{svi}).
	\begin{itemize}
		\item[(i)]  There exists  measurable variable $L(\xi)$  such that $\mathbb{E}[L^p(\xi)]<\infty$ for some $p\geq 1$ and
		\begin{equation}
		\label{lipschitz}
		\left\| F(x,\xi)- F\left(x^*,\xi\right)\right\| \leq L(\xi)\left\|x-x^*\right\| \quad  \forall x \in \C.
\end{equation}
There exist constants $C$ and $\varepsilon >0$ such that for $x \in \C \cap\left\{x:\left\|x-x^*\right\| \leq \varepsilon\right\}$
		\begin{equation}
		\label{qu}
		\left\| f(x)- f\left(x^*\right)-\nabla f\left(x^*\right)\left(x-x^*\right)\right\| \leq C\left\|x-x^{*}\right\|^{2}.
		\end{equation}
		\item[(ii)] The vector $ f\left(x^*\right)$ satisfies
		\[
		- f\left(x^*\right) \in \operatorname{ri} \mathcal{N}_{\C}\left(x^*\right),
		\]
		{where $\operatorname{ri} \mathcal{N}_{\C}\left(x^*\right)$ is the relative interior of normal cone $ \mathcal{N}_{\C}\left(x^*\right)$ \cite[Definition 6.3]{1998Variational}.}
		\item[(iii)] There exists $\mu>0$ such that for any $x \in \mathcal{T}_{\C}\left(x^*\right)$,
		\[
		x^{T}\nabla f\left(x^*\right) x \geq \mu\|x\|^{2},
		\]
		where $\mathcal{T}_{\C}\left(x^*\right)$ is  the critical tangent set to $\C$ at $x^*$, that is,
		\begin{equation}
		\label{eq:mathcal{T}}
		\mathcal{T}_{\C}\left(x^*\right):= \{x\in\R^n:Ax=0\}.
		\end{equation}
		
	\item[(iv)]
	The  covariance matrix $\operatorname{Cov}\left( F\left(x^* , \xi\right)\right)$ is finite.
 	\end{itemize}
\end{assu}

Condition (\ref{lipschitz}) in {Assumption \ref{assu_1}  is} the calmness of $F(\cdot,\xi)$ at point $x^*$ relative to $\C$, which implies the calmness of $f(\cdot)$ at point $x^*$, that is,
$$\left\| f(x)- f\left(x^*\right)\right\| \leq L\left\|x-x^*\right\| \quad \forall x \in \C,$$
 where $L = \mathbb{E}[L(\xi)]$.
% For the
% proof of Theorems \ref{xn-x*}, \ref{thm:non-ave}, \ref{T3.1} and \ref{T3.2}, the condition $p \geq 2$ in part (i) is sufficient.
 Condition (\ref{qu}) in {Assumption \ref{assu_1}  ensures} the boundedness of linear approximation error of $f(\cdot)$.
% , which will be used to  explore the consistency of batch-means estimator.
	Condition (ii) of Assumption \ref{assu_1}   is a constraint qualification which ensures the stability {of} the system of optimality conditions.  Condition (iii) of Assumption \ref{assu_1} means the   positive  definiteness  of $\nabla f(x^*)$ relative to subspace $\mathcal{T}_{\C}\left(x^*\right)$.

\begin{thm}
	\label{thm:SA-c}
Suppose that (i) Assumption \ref{assu_1} holds, (ii) step-size $\alpha_{k} =\alpha_{0} k^{-\beta}$ for some $\beta \in\left(\frac{1}{2}, 1\right)$ {and $ \alpha_{0}>0$}.  Then,
\begin{equation}
\label{eq:T2.1}
		\frac{1}{\sqrt{k}} \sum_{i=1}^{k}\left(x_{i}-x^*\right) \stackrel{d}{\to} \mathcal{N}\left(0, \mathrm{P}_{A}H^{\dagger} \mathrm{P}_{A} \Sigma \mathrm{P}_{A}H^{\dagger} \mathrm{P}_{A}\right)
\end{equation}
{with $k\to\infty$,}
where
\begin{equation*}
\label{P_A}
\Sigma:=\operatorname{Cov}\left( F\left(x^* , \xi\right)\right),\quad   H:=\nabla f\left(x^*\right),\quad\mathrm{P}_{A}:=\mathbf{I}_n-A^{T}\left(A A^{T}\right)^{\dagger} A.
\end{equation*}
\end{thm}

\begin{proof}
The asymptotic normality of SDA  for optimization problem has been studied in \cite[Theorem 4]{Duchi19}.
We just need to verify the conditions of \cite[Theorem 4]{Duchi19}(\cite[Assumption A-D]{Duchi19}).
Indeed, Assumption \ref{assu_1} (i)-(iii) are {variants}  of the conditions in \cite[Assumption A-C]{Duchi19}.
Combining Assumption \ref{assu_1} (i) and (iv),
we verify the condition of \cite[Assumption D]{Duchi19}.
The proof is complete.
\end{proof}

   Theorem \ref{thm:SA-c} shows the asymptotic normality  of
   Polyak-Ruppert averaged
 SDA for SVIP (\ref{svi}), which paves the way to construct the  confidence regions of the true solution to SVIP (\ref{svi}) by the  average of the iterates of SDA.
 However, {if   the  initial point  of SDA is far away from the solution and the iteration $k$ is not large enough}, the average of the iterates  may deviate from the true solution. This motivates us to study the
 asymptotic normality  of    last iterate of SDA for SVIP (\ref{svi}). {For {ease} of presentation, we assume the boundedness of} set $\C$.

 \begin{assu}
	\label{ass-batch}
	The set $\C$ is bounded.
\end{assu}

If $f(\cdot)$  is strictly monotone on $\C$ and Assumption 2.1 holds, Assumption \ref{ass-batch} is not necessary. Specifically, the solution $x^*$ of SVIP (\ref{svi}) must be the unique solution to  the new SVIP where   $\C$  is replaced
 {by} a bounded  convex set $\bar \C$ such that $x^*\in \bar \C$ \cite[Theorem 2.3.3]{VI-book03}.

The following theorem analyzes  the convergence  rate of last iterate $x_k$  to solution $x^*$, which plays {a} key role in  asymptotic normality  of    last iterate  of SDA for SVIP (\ref{svi}).

\begin{thm}
	\label{xn-x*}
	Suppose that (i) Assumptions \ref{assu_1} and \ref{ass-batch} hold, (ii) step-size
	{$\alpha_{k} =\alpha_{0} k^{-\beta} \text { with }
	\beta  \in\left(\frac{2}{3}, 1\right)$} and $ \alpha_{0}>0$. Then for any $\delta\in(0,1-\frac{1}{2\beta})$,
	\begin{equation}
	\label{rate}
	\|x_k-x^*\|=o(\alpha_{k}^\delta)~~~a.s.
	\end{equation}
\end{thm}

\begin{proof}
We employ Lemma \ref{lem:rate} in Appendix to study (\ref{rate}). We reformulate the recursion (\ref{recursion}) of Algorithm 1 into the form of iteration (\ref{linear reccursion}) in Lemma \ref{lem:rate} first.

%Recall the iteration of SDA (\ref{recursion}),
%$$
%x_{k}=\inmat{argmin}_{x\in\C}\left\{\left\langle z_{k-1}, x\right\rangle+\frac{1}{2}\|x\|_{2}^{2}\right\},
%$$
{
	Considering the KKT (Karush-Kuhn-Tucker) conditions of problem (\ref{recursion}) at $k$-th iteration and let $\lambda_{k-1}\geq  0$ and $\mu_{k-1}\geq 0$ be the corresponding lagrange multipliers.  It is easy to show that}
\begin{equation*}
x_{k+1}=x_{k}-\alpha_{k} F(x_k,\xi_k)+A^{T}\left(\lambda_{k-1}-\lambda_{k}\right)+D^{T}\left(\mu_{k-1}-\mu_{k}\right) .
\end{equation*}
Then
\begin{equation}
\label{x-x}
\begin{aligned}
\mathrm{P}_{A}\left(x_{k+1}-x^*\right) &=\mathrm{P}_{A}\left(x_{k}-x^*\right)-\alpha_{k} \mathrm{P}_{A} F(x_k,\xi_k)+\mathrm{P}_{A} D^{T}\left(\mu_{k-1}-\mu_{k}\right). \\
\end{aligned}
\end{equation}
Denote
\begin{equation}
\label{de:J}
\begin{cases}
J:=P_A\nabla f(x^*)P_A,\\
\Delta_k:=P_A(x_k-x^*),\\
S_{k}:= -P_{A}[F(x_k,\xi_k)-f(x_k)],\\
\zeta_{k}:=- P_{A}[f\left(x_{k}\right)- f\left(x^*\right)-\nabla f\left(x^*\right)\left(x_{k}-x^*\right)],\\
\epsilon_{k}:=\frac{1}{\alpha_{k}} [\mathrm{P}_{A} D^{T}\left(\mu_{k-1}-\mu_{k}\right)-\alpha_{k} \mathrm{P}_{A} \nabla f\left(x^*\right)\left(\mathbf{I}_n-\mathrm{P}_{A}\right)\left(x_{k}-x^*\right)].
\end{cases}
\end{equation}
We may reformulate the recursion (\ref{x-x}) as	
	\begin{equation}
	\label{Delta_k1}
	\Delta_{k+1}=\left(\mathbf{I}_n-\alpha_{k}J\right) \Delta_{k}+\alpha_{k}( \zeta_{k}+S_k+\epsilon_{k}).
	\end{equation}
	Let $
	D_{k}=-\zeta_{k} \frac{\Delta_{k}^{T}}{\left\|\Delta_{k}\right\|^{2}}$,
(\ref{Delta_k1}) can be rewritten as
	\begin{equation}
	\label{Delta_k2}
	\Delta_{k+1}=\left[\mathbf{I}_n-\alpha_{k}\left(J+D_{k}\right)\right] \Delta_{k}+\alpha_{k}\left(S_{k}+\epsilon_{k}\right).
	\end{equation}
	Dividing $\alpha_{k+1}^{\delta}$ on both sides of equation (\ref{Delta_k2}),
	\begin{equation}
	\begin{aligned}
	\label{2.13}
	\frac{\Delta_{k+1}}{\alpha_{k+1}^{\delta}} &=\left(\frac{\alpha_{k}}{\alpha_{k+1}}\right)^{\delta}\left[\mathbf{I}_n-\alpha_{k}\left(J+D_{k}\right)\right] \frac{\Delta_{k}}{\alpha_{k}^{\delta}}+\alpha_{k}\left(\frac{S_{k}}{\alpha_{k+1}^{\delta}}+\frac{\epsilon_{k}}{\alpha_{k+1}^{\delta}}\right) \\
	&=\left[\mathbf{I}_n-\alpha_{k}\left(J+C_{k}\right)\right] \frac{\Delta_{k}}{\alpha_{k}^{\delta}}+\alpha_{k}\left(\frac{S_{k}}{\alpha_{k+1}^{\delta}}+\frac{\epsilon_{k}}{\alpha_{k+1}^{\delta}}\right),
%	 \\
%	&=\left[\mathbf{I}_n_{d}-\alpha_{k} J_{k}\right] \frac{\Delta_{k}}{\alpha_{k}^{\delta}}+\alpha_{k}\left(\frac{S_{k}}{\alpha_{k+1}^{\delta}}+\frac{\epsilon_{k}}{\alpha_{k+1}^{\delta}}\right),
	\end{aligned}
	\end{equation}
	where
	\begin{equation}
	\label{C_k}
	C_{k}=\frac{1}{\alpha_{k}}\left(1-\left(\frac{\alpha_{k}}{\alpha_{k+1}}\right)^{\delta}\right) \mathbf{I}_n+\left(\left(\frac{\alpha_{k}}{\alpha_{k+1}}\right)^{\delta}-1\right) J+\left(\frac{\alpha_{k}}{\alpha_{k+1}}\right)^{\delta} D_{k}.
	\end{equation}
By the definitions of $\Delta_{k}, J$  in (\ref{de:J}) and  the fact $D_{k}=-\zeta_{k} \frac{\Delta_{k}^{T}}{\left\|\Delta_{k}\right\|^{2}}$
	$$
	\Delta_{k}=P_{A} \Delta_{k},\quad J=P_{A} J, \quad D_{k}=P_{A} D_{k},
	$$	
which induce
	$$
	\begin{aligned}
%	J_{k} \frac{\Delta_{k}}{\alpha_{k}^{\delta}} &=
	\left(J+C_{k}\right) \frac{\Delta_{k}}{\alpha_{k}^{\delta}}
	&=\frac{1}{\alpha_{k}}\left(1-\left(\frac{\alpha_{k}}{\alpha_{k+1}}\right)^{\delta}\right) \frac{\Delta_{k}}{\alpha_{k}^{\delta}}+\left(\frac{\alpha_{k}}{\alpha_{k+1}}\right)^{\delta} J \frac{\Delta_{k}}{\alpha_{k}^{\delta}}+\left(\frac{\alpha_{k}}{\alpha_{k+1}}\right)^{\delta} D_{k} \frac{\Delta_{k}}{\alpha_{k}^{\delta}} \\
%	&=\frac{1}{\alpha_{k}}\left(1-\left(\frac{\alpha_{k}}{\alpha_{k+1}}\right)^{\delta}\right) \frac{P_{A} \Delta_{k}}{\alpha_{k}^{\delta}}+\left(\frac{\alpha_{k}}{\alpha_{k+1}}\right)^{\delta} P_{A} J \frac{\Delta_{k}}{\alpha_{k}^{\delta}}+\left(\frac{\alpha_{k}}{\alpha_{k+1}}\right)^{\delta} P_{A} D_{k} \frac{\Delta_{k}}{\alpha_{k}^{\delta}} \\
	&=P_{A}\left(\frac{1}{\alpha_{k}}\left(1-\left(\frac{\alpha_{k}}{\alpha_{k+1}}\right)^{\delta}\right) \mathbf{I}_n+\left(\frac{\alpha_{k}}{\alpha_{k+1}}\right)^{\delta} J+\left(\frac{\alpha_{k}}{\alpha_{k+1}}\right)^{\delta} D_{k}\right) \frac{\Delta_{k}}{\alpha_{k}^{\delta}}\\
	&=P_{A} (J+C_{k} )\frac{\Delta_{k}}{\alpha_{k}^{\delta}}.
	\end{aligned}
	$$
	Then (\ref{2.13}) can be rewritten as
\begin{equation}
	\label{39}
	\frac{\Delta_{k+1}}{\alpha_{k+1}^{\delta}}=\left[\mathbf{I}_n-\alpha_{k} P_{A} (J+C_{k})\right] \frac{\Delta_{k}}{\alpha_{k}^{\delta}}+\alpha_{k}\left(\frac{S_{k}}{\alpha_{k+1}^{\delta}}+\frac{\epsilon_{k}}{\alpha_{k+1}^{\delta}}\right).
	\end{equation}
{Let $\Lambda$ be the orthogonal matrix with the set of eigenvectors associated with projection matrix $P_A$, and $\left(\begin{array}{cc}
		\mathbf{I}_r & \mathbf{0} \\
		\mathbf{0} & \mathbf{0}
	\end{array}\right)$
	being the associated diagonal matrix of eigenvalues,}
%	{Let $\Lambda$ be the orthogonal matrix to eigendecomposition of projection matrix $P_A$ with diagonal matrix $\left(\begin{array}{cc}
%	\mathbf{I}_r & \mathbf{0} \\
%	\mathbf{0} & \mathbf{0}
%	\end{array}\right)$,}
 $\left(\Lambda^{T}\right)^{(r)}$ be a $r \times n$-matrix composed of first $r$ row vectors of $\Lambda^{T}$ and $G_{k}$ be the $r$-order leading {principle submatrix}
  of $\Lambda^{T} (J+C_{k}) \Lambda$. Denote
	\begin{equation}
	\label{41}
	\Delta_{k}^{\prime}=\left(\Lambda^{T}\right)^{(r)} \Delta_{k}, \quad  S_{k}^{\prime}=\left(\Lambda^{T}\right)^{(r)} S_{k}, \quad \epsilon_{k}^{\prime}=\left(\Lambda^{T}\right)^{(r)} \epsilon_{k}.
	\end{equation}
	Then by \cite[Lemma 4]{S_Zhao}, (\ref{39}) can be rewritten as
%	
%	Notice that the identity of  subspace for SDA by \cite[Theorem 3]{Duchi19},  by \cite[Lemma 4]{S_Zhao} there exists an invertible matrix $\Lambda$, such that
\begin{equation}
\label{40}
	\left(\begin{array}{c}
	\frac{\Delta_{k+1}^{\prime}}{\alpha_{k+1}^{\delta}} \\
	\mathbf{0}
	\end{array}\right)=\left(\begin{array}{c}
	\frac{\Delta_{k}^{\prime}}{\alpha_{k}^{\delta}} \\
	\mathbf{0}
	\end{array}\right)-\alpha_{k}\left(\begin{array}{c}
	G_{k} \frac{\Delta_{k}^{\prime}}{\alpha_{k}^{\delta}} \\
	\mathbf{0}
	\end{array}\right)+\alpha_{k}\left[\left(\begin{array}{c}
	S_{k}^{\prime} \\
	\hline\alpha_{k+1}^{\delta} \\
	\mathbf{0}
	\end{array}\right)+\left(\begin{array}{c}
	\epsilon_{k}^{\prime} \\
	\hline \alpha_{k+1}^{\delta} \\
	\mathbf{0}
	\end{array}\right)\right].
\end{equation}
%	where $\Lambda$ is the orthogonal matrix to eigendecomposition of projection matrix $P_A$ with diagonal matrix $\left[\begin{array}{ll}
%		\mathbf{I}_r & 0 \\
%		0 & 0
%	\end{array}\right]$,
%	\begin{equation}
%	\label{41}
%	\Delta_{k}^{\prime}=\left(\Lambda^{T}\right)^{(r)} \Delta_{k}, \quad  S_{k}^{\prime}=\left(\Lambda^{T}\right)^{(r)} S_{k}, \quad \epsilon_{k}^{\prime}=\left(\Lambda^{T}\right)^{(r)} \epsilon_{k},
%	\end{equation}
%	$\left(\Lambda^{T}\right)^{(r)}$ is a $r \times n$ -matrix composed of first $r$ row vectors of $\Lambda^{T}$ and $G_{k}$ is the $r$ -order leading {\color{blue}principle submatrix} of $\Lambda^{T} (J+C_{k}) \Lambda$.
Obviously, it is sufficient to focus on  the nonzero part of (\ref{40}),
	\begin{equation}
	\label{42}
	\frac{\Delta_{k+1}^{\prime}}{\alpha_{k+1}^{\delta}}=\left(\mathbf{I}_{r}-\alpha_{k} G_{k}\right) \frac{\Delta_{k}^{\prime}}{\alpha_{k}^{\delta}}+\alpha_{k}\left(\frac{S_{k}^{\prime}}{\alpha_{k+1}^{\delta}}+\frac{\epsilon_{k}^{\prime}}{\alpha_{k+1}^{\delta}}\right).
	\end{equation}
	Setting
	$$
	y_{k}=\frac{\Delta_{k}^{\prime}}{\alpha_{k}^{\delta}}, \quad F_{k}=-G_{k}, \quad e_{k}=\frac{S_{k}^{\prime}}{\alpha_{k+1}^{\delta}}, \quad v_{k}=\frac{\epsilon_{k}^{\prime}}{\alpha_{k+1}^{\delta}},
	$$
	(\ref{42})
%	 can be rewritten as
%	$$
%	y_{k+1}=y_{k}+\alpha_{k} F_{k} y_{k}+\alpha_{k}\left(e_{k}+v_{k}\right),
%	$$
%which
is exact the formulation (\ref{linear reccursion}) in Lemma \ref{lem:rate}.

In what follows, we verify the conditions of Lemma \ref{lem:rate}.
	Firstly, we show that $-G_{k}$ converges to a stable matrix \footnote{All the eigenvalues of the matrix have strictly negative real part.}.
	Recall the definition (\ref{C_k}),  the first two terms in $C_k$
		$$
		\left(\frac{\alpha_{k}}{\alpha_{k+1}}\right)^{\delta} \rightarrow 1, \quad \frac{1}{\alpha_{k}}\left(1-\left(\frac{\alpha_{k}}{\alpha_{k+1}}\right)^{\delta}\right)=\frac{k^{\beta}}{\alpha_{0}}\left(1-\left(1+\frac{1}{k}\right)^{\beta \delta}\right) \rightarrow 0,
		$$
	as  {$\alpha_{k} =\alpha_{0} k^{-\beta}, \beta \in(2 / 3,1).
	$}
	Moreover, for large enough $k$, the third term of $C_k$ satisfies
		$$
		\left\|D_{k}\right\|  \leq \frac{C\left\|P_{A}\right\|\left\|{x}_{k}-x^{*}\right\|^{2}}{\left\|{x}_{k}-x^{*}\right\|}=C\left\|P_{A}\right\|\left\|{x}_{k}-x^{*}\right\|,
		$$
		where the   inequality follows from  the definition of $\zeta_{k}$ and  (\ref{qu}).  Then $C_{k} \rightarrow 0$ almost surely
	as ${x}_{k} \rightarrow x^*$ almost surely \cite[Theorem 2]{Duchi19}.
	Combining the fact that $G_{k}$ is the $r$-order leading {principle submatrix} of $\Lambda^{T} (J+C_{k}) \Lambda$,  $G_{k}$ converges to  the $r$-order leading {principle submatrix} of $\Lambda^{T} J \Lambda$, which is a positive definite matrix by \cite[Lemma 4]{S_Zhao}. Then, the limit of $\left\{-G_{k}\right\}$ is stable.

Next, we show $
\frac{\epsilon_{k}^{\prime}}{\alpha_{k+1}^{\delta}} \rightarrow 0
$ almost surely.
%, which is sufficient to show
%	$$
%	\frac{\epsilon_{k}^{\prime}}{\alpha_{k+1}^{\delta}} \rightarrow 0.
%	$$
	Recall the definition of $\epsilon_{k}$,
	$$
	\epsilon_{k}=\frac{1}{\alpha_{k}}[ P_{A} D^{T}\left(\mu_{k-1}-\mu_{k}\right)- {\alpha_{k}}P_{A} \nabla f\left(x^*\right)\left(P_{A}-\mathbf{I}_n\right)\left({x}_{k}-x^*\right)].
	$$
	By \cite[Theorem 3]{Duchi19}, $\epsilon_{k}=0$  when $k$ is large enough as $\mu_{k}=\mu_{k+1}=0$ and $\left(P_{A}-\mathbf{I}_n\right)\left({x}_{k}-x^*\right)=0 .$ Then $\frac{\epsilon_{k}^{\prime}}{\alpha_{k+1}^{\delta}}=\frac{\left(\Lambda^{T}\right)^{(r)} \epsilon_{k}}{\alpha_{k+1}^{\delta}}\to0$ almost surely.
	
	We verify
	$$
	\sum_{k=1}^{\infty} \frac{\alpha_{k}S_{k}^{\prime} }{\alpha_{k+1}^{\delta}}<\infty \quad \text { a.s. }
	$$
	Denote
	$$
	e_{k}^{\prime}=\left(\frac{\alpha_{k}}{\alpha_{k+1}}\right)^{\delta}\left(\Lambda^{T}\right)^{(r)} S_{k}.
	$$
	{Define the filtration \begin{equation}
	\label{eq:filtration}
	\mathcal{F}_{k}:=\sigma(S_1,\cdots,S_{k-1}),
	\end{equation}}where {$\sigma(S_1,\cdots,S_{k-1})$ is the $\sigma$-algebra generated by $\{S_1,\cdots,S_{k-1}\}$.}
	Obviously, $\left\{e_{k}^{\prime}, \mathcal{F}_{k+1}\right\}$ is a martingale difference sequence as $\left\{S_{k}, \mathcal{F}_{k+1}\right\}$ is. Then,
\begin{equation}
\label{e_k'}
	\begin{aligned}
	\sup _{k} \mathbb{E}\left[\left\|e_{k}^{\prime}\right\|^{2} \mid \mathcal{F}_{k}\right] &=\sup _{k} \mathbb{E}\left[\left\|\left(\frac{\alpha_{k}}{\alpha_{k+1}}\right)^{\delta}\left(\Lambda^{T}\right)^{(r)} S_{k}\right\|^{2} \mid \mathcal{F}_{k}\right] \\
	& \leq \sup _{k}\left(\frac{\alpha_{k}}{\alpha_{k+1}}\right)^{\delta}\left\|\left(\Lambda^{T}\right)^{(r)}\right\|^{2} \mathbb{E}\left[\left\|S_{k}\right\|^{2} \mid \mathcal{F}_{k}\right] \\
	& \leq 4^{\delta}\left\|\left(\Lambda^{T}\right)^{(r)}\right\|^{2} \sup _{k} \mathbb{E}\left[\left\|S_{k}\right\|^{2} \mid \mathcal{F}_{k}\right], \\
%	&=4^{\delta}\left\|\left(\Lambda^{T}\right)^{(r)}\right\|^{2} \sup _{k} \mathbb{E}\left[\left\| P_{A} S_{k}\right\|^{2} \mid \mathcal{F}_{k}\right] \\
%	& \leq 4^{\delta}\left\|\left(\Lambda^{T}\right)^{(r)}\right\|^{2}\left\|P_{A}\right\|^{2} \sup _{k} \mathbb{E}\left[\left\|S_{k}\right\|^{2} \mid \mathcal{F}_{k}\right]
	% \\
	%& \leq 4^{\delta}\left\|\left(\Lambda^{T}\right)^{(r)}\right\|^{2}\left\|P_{A}\right\|^{2} 4 \underline{L_{0}^{2}}
	\end{aligned}
\end{equation}
	where the second inequality follows from
	$$
	\left(\frac{\alpha_{k}}{\alpha_{k+1}}\right)^{\delta}=\left(1+\frac{1}{k}\right)^{\beta \delta} \leq 2^{\beta \delta}.
	$$
%	and the third inequality follows from the convexity of $\|\cdot\|^{2}$.
%	By the definition of $S_k$,
%	$$
%	\begin{aligned}
%	S_k=F(x_k,\xi_k)-f(x_k)=S_{k,1}+S_{k,2},
%	\end{aligned}
%	$$
{Define $S_{k,1}:=P_A[F(x_k,\xi_k)-F(x^*,\xi_k)+f(x^*)-f(x_k)]$ and $S_{k,2}:=P_A[F(x^*,\xi_k)-f(x^*)]$.
Obviously, $$
\mathbb{E}\left[\left\|S_{k,2}\right\|^{2} \mid \mathcal{F}_{k}\right]
=
\|\Sigma\|.
$$
Moreover,  Assumption \ref{assu_1} (i) implies
$$
\mathbb{E}\left[\left\|S_{k,1}\right\|^{2} \mid \mathcal{F}_{k}\right]
\leq
4L^2\|\Delta_k\|^2,
$$
and Assumption \ref{ass-batch} implies
\begin{equation}
\label{eq:S_k,1}
\mathbb{E}\left[\left\|S_{k}\right\|^{2} \mid \mathcal{F}_{k}\right]
=
\mathbb{E}\left[\left\|S_{k,1}+S_{k,2}\right\|^{2} \mid \mathcal{F}_{k}\right]
\leq
\|\Sigma\|+4L^2\|\Delta_k\|^2+4L\|\Sigma\|^{\frac{1}{2}}\|\Delta_k\|
< \infty.
\end{equation}}
 Then, (\ref{e_k'}) is finite.
 Since
	$$
	\sum_{k=1}^{\infty} \alpha_{k}^{2(1-\delta)}=\sum_{k=1}^{\infty} \frac{\alpha_{0}^{2(1-\delta)}}{k^{2(1-\delta) \beta}}<\infty,
	$$
the convergence theorem {of} martingale difference sequences \cite[Appendix {B}.6, Theorem B 6.1]{Chen06} ensures that
	$$
	\sum_{k=1}^{\infty} \alpha_{k}^{1-\delta} e_{k}^{\prime}<\infty,
	$$
	which implies
	$$
	\sum_{k=1}^{\infty}  \frac{\alpha_{k}S_{k}^{\prime}}{\alpha_{k+1}^{\delta}}=\sum_{k=1}^{\infty} \alpha_{k}^{1-\delta} e_{k}^{\prime}<\infty.
	$$
Subsequently, Lemma \ref{lem:rate} implies $\frac{\Delta_{k}^{\prime}}{\alpha_{k}^{\delta}} \rightarrow 0$ almost surely. By the definition of $\Delta_{k}^{\prime}$ in (\ref{41}),  $\left\|x_k-x^*\right\|=o\left(\alpha_{k}^{\delta}\right)$ almost surely. The proof is complete.
\end{proof}
We are ready to study the asymptotic normality  of the    last iterate  of SDA for SVIP (\ref{svi}).
%The following theorem  proposes the asymptotic normality  of   the average of the iterates of SDA for SVIP (\ref{svi}). This also provides a basis for constructing the confidence region in the case of non-ergodic iteration.
\begin{thm}
	\label{thm:non-ave}
	Suppose that (i) Assumptions \ref{assu_1} and \ref{ass-batch} hold, (ii) step-size
	{$\alpha_{k} =\alpha_{0} k^{-\beta} \text { with }
	\beta  \in\left(\frac{2}{3}, 1\right)$} and $ \alpha_{0}>0$,  (iii)
%	the covariance matrix mapping $\operatorname{Cov}\left(F\left(\cdot , \xi\right)\right)$ is continuous at point $x^*$ and
$\Lambda$ is the orthogonal matrix with the set of eigenvectors associated with projection matrix $P_A$, and $\left(\begin{array}{cc}
	\mathbf{I}_r & \mathbf{0} \\
	\mathbf{0} & \mathbf{0}
\end{array}\right)$
is the associated diagonal matrix of eigenvalues. Then
\begin{equation}
\label{normal}
	\frac{x_{k}-x^*}{\sqrt{\alpha_{k}}} \stackrel{d}{\longrightarrow} \mathcal{N}(0, \tilde{\Sigma}),
\end{equation}
	where
\begin{equation}
\label{Sigma1}
	\begin{array}{c}
	\tilde{\Sigma}=\Lambda\left(\begin{array}{cc}
	\Sigma_{1} & \mathbf{0}\\
	\mathbf{0} & \mathbf{0}
	\end{array}\right) \Lambda^{T}, \quad
%	\quad 	\Lambda^{T} P_{A} \Lambda=\left(\begin{array}{cc}
%	I_{r} & \mathbf{0}_{1} \\
%	\mathbf{0}_{2} & \mathbf{0}_{3}
%	\end{array}\right),\\
	\Sigma_{1}=\int_{0}^{\infty} e^{(-G) t}\left(\Lambda^{T}\right)^{(r)} P_{A} {\Sigma} P_{A}\left( \left(\Lambda^{T}\right)^{(r)}\right) ^T e^{\left(-G^{T}\right) t} \mathrm{~d} t,\\
    \end{array}
\end{equation}
 	 $\left(\Lambda^{T}\right)^{(r)} \in \R^{r \times n}$ is composed by first $r$ row vectors of $\Lambda^{T}, G$ is the r-order leading {principle submatrix} of $\Lambda^{T} J \Lambda$ and $J=P_A \nabla f(x^*) P_A$.
\end{thm}
\begin{proof}
	{We mimic the proof of \cite[Theorem 3]{S_Zhao} to study (\ref{normal}).}
	 We employ  Lemma \ref{lem:asym norm} in Appendix to prove (\ref{normal}). We first reformulate $\Delta_{k}$  into the form of formula (\ref{linear reccursion_1}) in Lemma \ref{lem:asym norm}.
%	 By definition of $J$,
%	 	\begin{equation*}
%	 	J=P_A\nabla f(x^*)P_A=P_AJ.
%	 	\end{equation*}
	 	
%	 	Substitute (\ref{T}) into (\ref{Delta_k1}),
%	 	\begin{equation}\label{Delta_k1}
%	 \begin{aligned}
%	 	\Delta_{k+1}
%	 	&=\left[I-\alpha_{k}P_AJ\right]\Delta_k+\alpha_k\left(S_k-\zeta_k+\epsilon_k\right).
%	 \end{aligned}
%	 	\end{equation}
	 	Left  multiplying $\Lambda^T$ on both side of  (\ref{Delta_k1}), we have by
	 	\cite[Lemma 4]{S_Zhao} that
	 	\begin{equation}
	 	\label{2.19}
	 	\left(
	 	\begin{array}{cc}
	 	\Delta_{k+1}^{'}\\
	 	\textbf{0}\\
	 	\end{array}\right)=
	 	\left(
	 	\begin{array}{cc}
	 	\Delta_{k}^{'}\\
	 	\textbf{0}\\
	 	\end{array}\right)
	 	-\alpha_k\left(
	 	\begin{array}{cc}
	 	G\Delta_{k}^{'}\\
	 	\textbf{0}\\
	 	\end{array}\right)+\alpha_{k}\left[
	 	\left(
	 	\begin{array}{cc}
	 	\zeta_k^{'}\\
	 	\textbf{0}\\
	 	\end{array}\right)+
	 	\left(
	 	\begin{array}{cc}
	 	S_k^{'}\\
	 	\textbf{0}\\
	 	\end{array}\right)+
	 	\left(
	 	\begin{array}{cc}
	 	\epsilon_k^{'}\\
	 	\textbf{0}\\
	 	\end{array}\right)
	 	\right],
	 	\end{equation}
	 	where $G$ is the $r$-order leading {principle submatrix} of $\Lambda^{T} J \Lambda$,
	 	\begin{equation*}
	 	\label{S_k'}
	 	\Delta_{k}^{'}=(\Lambda^T)^{(r)}\Delta_{k},\quad\zeta_k^{'}=(\Lambda^T)^{(r)}\zeta_k,\quad S_k^{'}=(\Lambda^T)^{(r)}S_k,\quad\epsilon_k^{'}=(\Lambda^T)^{(r)}\epsilon_k.
	 	\end{equation*} 	
	 	Obviously, it is sufficient to focus on  the nonzero part of (\ref{2.19}),
	 	\begin{equation}\label{recursion 5}
	 	\Delta_{k+1}^{'} =(\mathbf{I}_r-\alpha_{k}G)\Delta_{k}^{'}+\alpha_{k}\left(\zeta_k^{'}+S_k^{'}+\epsilon_k^{'}\right).
	 	\end{equation}
Setting
	 	\begin{equation*}
	 	y_k=\Delta_{k}^{'},\quad F_k=-G, \quad e_k=S_k^{'},\quad \upsilon_k=\zeta_k^{'}+\epsilon_k^{'},
	 	\end{equation*}
	 	(\ref{recursion 5})
%	 	can be rewritten as
%	 	\begin{equation*}
%	 	y_{k+1}=y_k+\alpha_{k}F_ky_k+\alpha_{k}\left(e_k+\tau_k\right),
%	 	\end{equation*}
%	 which
	 is exact the formulation (\ref{linear reccursion_1}) in Lemma \ref{lem:asym norm}.
	 	
	 	Next, we verify the conditions of Lemma \ref{lem:asym norm}. By the setting of step-size $\alpha_{k}$,
$
	 	\alpha_{k+1}^{-1}-\alpha_{k}^{-1}\rightarrow 0,
$
	 	which implies condition (i) of Lemma \ref{lem:asym norm}.
	 	By the definition of $G$, $-G$  is stable, condition (ii) of Lemma \ref{lem:asym norm} holds.
	 	In what follows, we verify condition (iii) of    Lemma \ref{lem:asym norm}. We first show that $\epsilon_k^{'}+\zeta_k^{'}=o(\sqrt{\alpha_k})$ almost surely.
%	 	In fact, for $\epsilon_k^{'}$, recall the definition of $\epsilon_k$ in (\ref{T}).
	
	By  \cite[Theorem 3]{Duchi19}, $\epsilon_k=0$ almost surely for $k$ large enough and then $\epsilon_k^{'}=(\Lambda^T)^{(r)}\epsilon_k=0$ almost surely for $k$ large enough.
	 	By the definition of $\zeta_k^{'}$,
	 {	\begin{equation*}
	 	\begin{aligned}
	 	\|\zeta_k^{'}\|&=\left\|-(\Lambda^T)^{(r)}P_A\left[  f(x_k)-  f(x^*)-\nabla f(x^*)(x_k-x^*)\right]\right\|\\
	 	&\le {C}\left\|(\Lambda^T)^{(r)}P_A\right\|\left\|x_k-x^*\right\|^2
	 	=o\left(\alpha_k^{2\delta}\right)\quad \text{a.s.},
	 	\end{aligned}
	 	\end{equation*}}where the  inequality follows from Assumption \ref{assu_1} (i) and the last equality follows from Theorem \ref{xn-x*}. Therefore,
	 	\begin{equation*}
	 	\label{e+zeta}
	 	\epsilon_k^{'}+\zeta_k^{'}=o\left(\alpha_k^{2\delta}\right)\le o(\sqrt{\alpha_k})\quad \text{a.s.,}
	 	\end{equation*}
	 	{as we may choose} $\delta\in (1/4,1-{1}/{(2\beta)})$.
	 	By mimicking the proof of \cite[(57)-(61)]{S_Zhao},
 the conditions (\ref{c1}-\ref{c3}) in Lemma \ref{lem:asym norm} hold.
	 	
	 	Summarizing above, all the conditions of Lemma \ref{lem:asym norm} hold. Then, %\cite[Theorem 3.3.1]{Chen06}
	 	\begin{equation*}
	 	\dfrac{\Delta_{k}^{'}}{\sqrt{\alpha_{k}}}\xrightarrow{d}\mathcal{N}(0,\Sigma_1),
	 	\end{equation*}
	 	where $\Sigma_1$ is defined in (\ref{Sigma1}).
	 	%	\begin{equation}
	 	%	\Sigma_1=\int_{0}^{\infty}e^{(-G)t}(U^T)^{(r)}P_A\bar{\Sigma}P_A(U^T)^{(r)T}e^{(-G^T)t}\d t,\quad\bar{\Sigma}=\dfrac{1}{m^2}\sum_{j=1}^m\operatorname{Cov}(  F_j(x^*,\xi_j)),
	 	%	\end{equation}
%	 	and $(U^T)^{(r)}\in\R^{r\times n}$ is composed by first $r$ row vectors of $U^T$.
	 	Note that $\Delta_{k}=\Lambda\left((\Delta_{k}^{'})^T,\textbf{0}^T\right)^T$ and by the definition of $\tilde{\Sigma}$   in (\ref{Sigma1}),
	 	$$
	 	\dfrac{\Delta_{k}}{\sqrt{\alpha_{k}}}\xrightarrow{d}\mathcal{N}(0,\tilde{\Sigma}),
	 	$$
	 	{which implies (\ref{normal}).}
	\end{proof}
Theorem \ref{thm:non-ave} presents the asymptotic normality of the    last iterate  of SDA for SVIP (\ref{svi}) with the rate $1/\sqrt{\alpha_{k}}$. Note that step-size
{$\alpha_{k} =\alpha_{0} k^{-\beta} \text { and }
\beta  \in\left(\frac{2}{3}, 1\right)$},
the convergence rate of the asymptotic normality  of the    last iterate can not arrive at $\sqrt{k}$.
 {Similarly},   Theorem \ref{thm:non-ave} ensure us to {construct the  confidence regions of the true solution to SVIP (\ref{svi}) by the last iterate  of SDA.}

\section{Estimator for the covariance matrix}
Inference is a core topic in statistics and the confidence region has been widely used to quantify the uncertainty in the estimation of model parameters. The asymptotic normality {of SDA} is the first step of building the confidence regions of the true solutions {for} SVIP (\ref{svi}). {Next}, we have to provide  {estimators} of the asymptotic covariance {matrices} in the limit
normal distributions. In the seminal work  \cite{Chen16}, Chen et al. propose two online methods plug-in and batch-means to estimate the covariance  matrix  when  vanilla SGD  is implemented to solve {unconstrained} stochastic optimization problems.
We extend the plug-in and batch-means methods to {SDA} algorithm for SVIP (\ref{svi}).

\subsection{Plug-in method}
Recall the normal distribution in Theorem \ref{thm:SA-c},
$$
 		\frac{1}{\sqrt{k}} \sum_{i=1}^{k}\left(x_{i}-x^*\right) \stackrel{d}{\to} \mathcal{N}\left(0, \mathrm{P}_{A}H^{\dagger} \mathrm{P}_{A} \Sigma \mathrm{P}_{A}H^{\dagger} \mathrm{P}_{A}\right).
$$
The idea of the plug-in  method \cite{Chen16} is to separately estimate $\Sigma$, $\mathrm{P}_{A}$ and $H^{\dagger}$ by some $\Sigma_k$, $\mathrm{P}_{A_k}$ and $H^{\dagger}_k$. However,  as the   Moore-Penrose inverse of matrix is not continuous, it is difficult to show the convergence of $H^{\dagger}_k$ to $H^{\dagger}$. This motivates us to reformulate the above normal {distribution through linear transformation first}.

Let $\Lambda$ be the orthogonal matrix with the set of eigenvectors associated with projection matrix $P_A$, and $\left(\begin{array}{cc}
	\mathbf{I}_r & \mathbf{0} \\
	\mathbf{0} & \mathbf{0}
\end{array}\right)$
being the associated diagonal matrix of eigenvalues,
 $\left(\Lambda^{T}\right)^{(r)}$ be a $r \times n$-matrix composed of first $r$ row vectors of $\Lambda^{T}$.
Left  multiplying $\Lambda^T$ {on} (\ref{eq:T2.1}), we have
\begin{equation*}
\Lambda^{T}P_A  \frac{1}{\sqrt{k}} \sum_{i=1}^{k}\left(x_{i}-x^*\right) \stackrel{d}{\to} \mathcal{N}\left(0, \Lambda^{T}\mathrm{P}_{A}H^{\dagger} \mathrm{P}_{A}\Lambda
\Lambda^{T} \mathrm{P}_{A}\Sigma\mathrm{P}_{A}\Lambda
\Lambda^{T} \mathrm{P}_{A}H^{\dagger} \mathrm{P}_{A}\Lambda\right).
\end{equation*}
By some {calculations} and the fact $(P_AHP_A)^{\dagger}=P_AH^{\dagger}P_A$
  \cite{Duchi19},
%$\Lambda$ is the orthogonal matrix,
$$
\begin{aligned}
&\Lambda^{T}\mathrm{P}_{A}H^{\dagger} \mathrm{P}_{A}\Lambda
=
(\Lambda^{T}\mathrm{P}_{A}H \mathrm{P}_{A}\Lambda)^{\dagger}
=
\left(\begin{array}{cc}
\left( \left(\Lambda^{T}\right)^{(r)}H\left( \left(\Lambda^{T}\right)^{(r)}\right) ^T\right) ^{-1}& \mathbf{0} \\
\mathbf{0} & \mathbf{0}
\end{array}\right),\\
&\Lambda^{T} \mathrm{P}_{A}\Sigma\mathrm{P}_{A}\Lambda=\left(\begin{array}{cc}
\left(\Lambda^{T}\right)^{(r)}\Sigma\left( \left(\Lambda^{T}\right)^{(r)}\right) ^T& \mathbf{0} \\
\mathbf{0} & \mathbf{0}
\end{array}\right).
\end{aligned}
$$
Note also that $x_k$   could identify the   subspace $\{x: Ax=b\}$ \cite[Theorem 3]{Duchi19},   (\ref{eq:T2.1}) can be rewritten as
$$
%		\left(
%	\begin{array}{cc}
%	{\sqrt{k}} \left(\Lambda^{T}\right)^{(r)}\left(\bar{x}-x^*\right)\\
%	\textbf{0}\\
%	\end{array}\right)
\frac{1}{\sqrt{k}} \sum_{i=1}^{k}\left(x_{i}-x^*\right)
\stackrel{d}{\to}
\mathcal{N}\left(0,
\Lambda\left(\begin{array}{cc}
\bar{H} ^{-1} \bar{\Sigma} \bar{H}^{-1}
& \mathbf{0} \\
\mathbf{0} & \mathbf{0}
\end{array}\right) \Lambda^{T}
\right)$$
{with $k\to\infty$, }
where \begin{equation*}
\label{eq:barSigma}
\bar{H}=\left(\Lambda^{T}\right)^{(r)}H\left( \left(\Lambda^{T}\right)^{(r)}\right) ^T,~\bar{\Sigma}=\left(\Lambda^{T}\right)^{(r)}\Sigma\left( \left(\Lambda^{T}\right)^{(r)}\right) ^T.
\end{equation*}
Then  the plug-in method {is to estimate} $\Lambda$, $\bar{\Sigma}$  and $\bar{H}$ separately.  Denote
  $A_k$ as the matrix with respect to active
 constraint on $x_k$,
 $$
 \mathrm{P}_{A_{k}}=\mathbf{I}_n-{A_{k}}^{T}\left({A_{k}} {A_{k}}^{T}\right)^{\dagger} {A_{k}},
 $$
 $$
\Sigma_k=\frac{1}{k}\sum_{i=1}^{k}F(x_{i-1},\xi_i)F(x_{i-1},\xi_i)^T-\left[ \frac{1}{k}\sum_{i=1}^{k}F(x_{i-1},\xi_i)\right] \left[ \frac{1}{k}\sum_{i=1}^{k}F(x_{i-1},\xi_i)\right] ^T$$
and
  $${H_k}=\frac{1}{k}\sum_{i=1}^{k}\nabla F(x_{i-1},\xi_i).$$
  Let $\Lambda_k$ be the orthogonal matrix with the set of eigenvectors associated with projection matrix $\mathrm{P}_{A_{k}}$, and $\left(\begin{array}{cc}
  	\mathbf{I}_{r_k} & \mathbf{0} \\
  	\mathbf{0} & \mathbf{0}
  \end{array}\right)$
  being the associated diagonal matrix of eigenvalues,
 $\left(\Lambda_k^{T}\right)^{({r_k})}$ be a ${r_k} \times n$-matrix composed of first ${r_k}$ row vectors of $\Lambda_k^{T}$.
 Then $\Lambda_k$,
 \begin{equation*}
 \label{eq:barH}
	\begin{aligned}
	\bar{\Sigma}_k:=\left(\Lambda_k^{T}\right)^{({r_k})}\Sigma_k\left( \left(\Lambda_k^{T}\right)^{({r_k})}\right) ^T,\quad
	\bar{H}_k:=\left(\Lambda_k^{T}\right)^{({r_k})}H_k\left( \left(\Lambda_k^{T}\right)^{({r_k})}\right) ^T
	\end{aligned}
 \end{equation*}
are the estimators of  $\Lambda, \bar \Sigma, \bar H$ respectively.

The consistency of the  plug-in estimator can be established under the following conditions.

\begin{assu}
	\label{plugass}
\begin{itemize}
	\item[(i)]	There exists measurable variable $L_2(\xi)$ such that $\mathbb{E}[L_2(\xi)]<\infty$ and
	$$\left\|\nabla F(x,\xi)-\nabla F(x^*,\xi)\right\|
	\leq {L_2}(\xi)\|x-x^*\|\quad \forall x\in\C.$$

	\item[(ii)] There exists a constant $C$ such that
 $F(x^*,\xi)\leq C$ almost surely.
\end{itemize}
\end{assu}
Assumption \ref{plugass} (i) is the calmness of $\nabla F(\cdot,\xi)$ at point $x^*$ relative to $\C$.
Assumption \ref{plugass} (ii) holds
if $F(x^*,\cdot)$ is continuous in $\xi$ and $\Xi$ is compact.

\begin{thm}
	\label{T3.1}
	Suppose that (i) Assumptions \ref{assu_1} and \ref{plugass} hold, (ii) step-size
	{$\alpha_{k} =\alpha_{0} k^{-\beta} \text { with }
	\beta  \in\left(\frac{2}{3}, 1\right)$} and $ \alpha_{0}>0$. Then
	\begin{equation*}
	\label{PSP}
	\left\| \Lambda_k\left(\begin{array}{cc}
	\bar{H}_k ^{\dagger} \bar{\Sigma}_k \bar{H}_k^{\dagger} & \mathbf{0} \\
	\mathbf{0} & \mathbf{0}
	\end{array}\right) \Lambda_k^{T}-\Lambda\left(\begin{array}{cc}
	\bar{H} ^{-1} \bar{\Sigma} \bar{H}^{-1} & \mathbf{0} \\
	\mathbf{0} & \mathbf{0}
	\end{array}\right) \Lambda^{T}\right\|\to 0~~~a.s.~~~
	\end{equation*}
%	where $\bar{H}_k$, $\bar{\Sigma}_k$, $\bar{H}$, $\bar{\Sigma}_k$ are defined in (\ref{eq:barSigma}) and (\ref{eq:barH}) respectively.
\end{thm}
\begin{proof}
%	For simplicity, we denote
%	$$F_s=\Sigma_k-\Sigma,~~~F_k=\mathrm{P}_{A_k}H_k^{\dagger} \mathrm{P}_{A_k},~~~F=\mathrm{P}_{A}H^{\dagger} \mathrm{P}_{A},~~~F_h=F_k-F.$$
%	Then,
%	The left hand of (\ref{PSP}) can be decomposed as follow
Following \cite[Theorem 3]{Duchi19}, SDA could identify the   subspace $\{x: Ax=b\}$, which implies
%	When the Assumption \ref{assu_1} is satisfied, for sufficiently large $k$,  the dual average algorithm identifies the   manifold with probability 1 by \cite[Theorem 3]{Duchi19}. In this sense, the difference of projection matrix almost surely tend to 0 for sufficiently large $k$,
\begin{equation*}
\label{P-P}
\mathrm{P}_{A_k}=\mathrm{P}_{A} ~~a.s.
\end{equation*}
for sufficiently large $k$.
As $\Lambda$ and $\Lambda_k$ are the orthogonal matrices to eigendecomposition of $P_A$ and $P_{A_k}$ with diagonal matrix $\left(\begin{array}{cc}
\mathbf{I}_{r} & \mathbf{0} \\
\mathbf{0} & \mathbf{0}
\end{array}\right)$ and $\left(\begin{array}{cc}
\mathbf{I}_{r_k} & \mathbf{0} \\
\mathbf{0} & \mathbf{0}
\end{array}\right)$ respectively, then
\begin{equation}
\label{eq:Lambda_k-Lambda}
r_k=r,\qquad\Lambda_k^{T}=\Lambda^T ~~a.s.
\end{equation}
for sufficiently large $k$.
Subsequently, 	
$$
\begin{aligned}
&\left\| \Lambda_k\left(\begin{array}{cc}
\bar{H}_k ^{\dagger} \bar{\Sigma}_k \bar{H}_k^{\dagger} & \mathbf{0} \\
\mathbf{0} & \mathbf{0}
\end{array}\right) \Lambda_k^{T}-\Lambda\left(\begin{array}{cc}
\bar{H} ^{-1} \bar{\Sigma} \bar{H}^{-1} & \mathbf{0} \\
\mathbf{0} & \mathbf{0}
\end{array}\right) \Lambda^{T}\right\| \\
&\leq
\left\| \Lambda_k\left(\begin{array}{cc}
\bar{H}_k ^{\dagger} \bar{\Sigma}_k \bar{H}_k^{\dagger}-\bar{H} ^{-1} \bar{\Sigma} \bar{H}^{-1} & \mathbf{0} \\
\mathbf{0} & \mathbf{0}
\end{array}\right) \Lambda_k^{T}\right\|
+2
\left\| \Lambda_k\right\| \left\|  \left(\begin{array}{cc}
\bar{H} ^{-1} \bar{\Sigma} \bar{H}^{-1} & \mathbf{0} \\
\mathbf{0} & \mathbf{0}
\end{array}\right)\right\| \left\| \Lambda_k^T-\Lambda^T\right\|.
\end{aligned}
$$	
Note that $\|\Lambda\|$, $\left\|H \right\| $ and $\|\Sigma\|$ are bounded and (\ref{eq:Lambda_k-Lambda}) holds,
we only need to study the consistency of $\left\|\bar{H}_k ^{\dagger} \bar{\Sigma}_k \bar{H}_k^{\dagger}-\bar{H} ^{-1} \bar{\Sigma} \bar{H}^{-1}\right\|  $.

%In what follows, we focus on the convergence of $\left\|\bar{H}_k ^{\dagger} \bar{\Sigma}_k \bar{H}_k^{\dagger}-\bar{H} ^{-1} \bar{\Sigma} \bar{H}^{-1}\right\|  $.
%By the submultiplicativity of matrix norm,
Obviously,
\begin{equation*}
\begin{aligned}
&\left\|\bar{H}_k ^{\dagger} \bar{\Sigma}_k \bar{H}_k^{\dagger}-\bar{H} ^{-1} \bar{\Sigma} \bar{H}^{-1}\right\|
%	=&
%	\left\|(F+F_h)(\Sigma+F_s)(F+F_h)-F\Sigma F\right\|\\
	\leq&
	\|\bar{H}_k ^{\dagger}\|^2\|\bar{\Sigma}_k-\bar{\Sigma}\|
	+
	\|\bar{H}_k ^{\dagger}-\bar{H} ^{-1}\|^2\|\bar{\Sigma}\|
	+
	2\|\bar{H} ^{-1}\|\|\bar{\Sigma} \|\|\bar{H}_k ^{\dagger}-\bar{H} ^{-1}\|.
\end{aligned}
\end{equation*}
	Note that
%	$\|F_k\|=\left\|H_k^{\dagger}\right\|_{2}\leq \lambda_{min}^{-1}(H_k)$,
	$\|\bar{H} ^{-1}\|$  and $\|\bar{\Sigma}\|$ are finite, it is sufficient to show  $\|\bar{H}_k ^{\dagger}-\bar{H} ^{-1}\| $ and $\|\bar{\Sigma}_k-\bar{\Sigma}\| $  converge to zero almost surely.
	By the definition of $H_k$,
	\begin{eqnarray}
	\nonumber\|H_k-H\|&=&\left\| \frac{1}{k}\sum_{i=1}^{k}\nabla F(x_{i-1},\xi_i)-H\right\| \\
	\label{hn-h}&\leq&\left\| \frac{1}{k}\sum_{i=1}^{k}\nabla F(x^*,\xi_i)-H\right\| +\left\| \frac{1}{k}\sum_{i=1}^{k}\left(\nabla F(x_{i-1},\xi_i)-\nabla F(x^*,\xi_i)\right)\right\| .
	\end{eqnarray}
As {$\xi_1,\xi_2,\cdots, \xi_k$} is iid sample, the strong law of large numbers ensures the first term on the right hand of (\ref{hn-h}) converges to zero almost surely.
	By {Assumption} \ref{plugass} (i),
	the second term on the right hand  of (\ref{hn-h})
\begin{equation*}
\begin{aligned}
	\label{3.36}\left\|\frac{1}{k}\sum_{i=1}^{k}\left(\nabla F(x_{i-1},\xi_i)-\nabla F(x^*,\xi_i)\right)\right\|
%	\leq&\frac{1}{k}\sum_{i=1}^{k}\left\|\nabla F(x_{i-1},\xi_i)-\nabla F(x^*,\xi_i)\right\|\\
\leq& \frac{1}{k}\sum_{i=1}^{k}L_2(\xi_{i})\|x_{i-1}-x^*\|,
\end{aligned}
\end{equation*}
which converges to zero as $x_{k}\to x^*$   almost surely  \cite[Theorem 2]{Duchi19}.
By the consistency of $H_k$ and  (\ref{eq:Lambda_k-Lambda}), $\bar{H}_k$ is nonsingular for sufficiently large $k$, that is, $\bar{H}_k^\dagger=\bar{H}_k^{-1}$.
Then $\|\bar{H}_k ^\dagger-\bar{H} ^{-1}\|\to 0$  almost surely  as $\|\bar{H}_k-\bar{H}\|\to 0$ almost surely.

Next, we study the convergence of $\|\bar{\Sigma}_k-\bar{\Sigma}\|$. By the definition of $\bar{\Sigma}_k$ and $\bar{\Sigma}$,
$$
\begin{aligned}
\left\| \bar{\Sigma}_k-\bar{\Sigma}\right\|
=&
\|\left(\Lambda_k^{T}\right)^{({r_k})}\Sigma_k\left( \left(\Lambda_k^{T}\right)^{({r_k})}\right) ^T-\left(\Lambda^{T}\right)^{(r)}\Sigma\left( \left(\Lambda^{T}\right)^{(r)}\right) ^T\|\\
\leq&
\|\left(\Lambda_k^{T}\right)^{({r_k})}\left( \Sigma_k-\Sigma\right) \left( \left(\Lambda_k^{T}\right)^{({r_k})}\right) ^T\|
+
2\|\left(\Lambda_k^{T}\right)^{({r_k})}\|\|\Sigma\|\|\left(\Lambda_k^{T}\right)^{({r_k})}-\left(\Lambda^{T}\right)^{(r)}\|.
\end{aligned}
$$
%	For $\bar{\Sigma}_k$, we decompose $F(x_{i-1},\xi_i)$ into the following terms:
%	$$F(x_{i-1},\xi_i)=F(x^*,\xi_i)+(F(x_{i-1},\xi_i)-F(x^*,\xi_i)).$$
Therefore, it is sufficient to show  $\|{\Sigma}_k-{\Sigma}\|\to 0$ almost surely.
	For easy of notation, we denote
	$$X_i:=F(x^*,\xi_i),~~~Y_i:=F(x_{i-1},\xi_i)-F(x^*,\xi_i),$$
	then,
	\begin{equation}
	\begin{aligned}
	\label{sigma}
	\|{\Sigma}_k-{\Sigma}\|
	=&
	\left\| \frac{1}{k}\sum_{i=1}^{k}(X_i+Y_i)(X_i+Y_i)^T-\left[ \frac{1}{k}\sum_{i=1}^{k}(X_i+Y_i)\right] \left[\frac{1}{k}\sum_{i=1}^{k}(X_i+Y_i)\right]^T-\Sigma\right\| \\
	\leq&
	\left\| \frac{1}{k}\sum_{i=1}^{k}X_iX_i^T-\left[\frac{1}{k}\sum_{i=1}^{k}X_i\right]\left[\frac{1}{k}\sum_{i=1}^{k}X_i\right]^T-\Sigma\right\|
	 +
	 \left\| \frac{1}{k}\sum_{i=1}^{k}Y_iY_i^T\right\|
	 \\
	 &
	 {+
	 \frac{2}{k}\sum_{i=1}^{k}\left\|X_iY_i^T\right\|
	 +\left\| \frac{1}{k}\sum_{i=1}^{k}Y_i\right\| ^2}
	 +
	 2\left\| \frac{1}{k}\sum_{i=1}^{k}X_i\right\| \left\| \frac{1}{k}\sum_{i=1}^{k}Y_i\right\|
	 .
	\end{aligned}
	\end{equation}
	Again, the strong law of large numbers implies the first term on the right hand of (\ref{sigma})  tends to zero almost surely.
	By Assumptions \ref{assu_1} (i) and  \ref{plugass} (ii), the last four terms on the right hand of (\ref{sigma})
			$$
\begin{aligned}
&\left\|\frac{1}{k}\sum_{i=1}^{k}Y_iY_i^T\right\|
%\leq\frac{1}{k}\sum_{i=1}^{k}\|F(x_{i-1},\xi_i)-F(x^*,\xi_i)\|^2
\leq \frac{1}{k}\sum_{i=1}^{k}L(\xi_i)^2\|x_{i-1}-x^*\|^2\to0~~~a.s.,\\
&\left\|\frac{1}{k}\sum_{i=1}^{k}Y_i\right\|
%\leq \frac{1}{k}\sum_{i=1}^{k}\|F(x_{i-1},\xi_i)-F(x^*,\xi_i)\|
\leq \frac{1}{k}\sum_{i=1}^{k}L(\xi_i)\|x_{i-1}-x^*\|\to0~~~a.s.,\\
&\frac{2}{k}\sum_{i=1}^{k}\left\|X_iY_i^T\right\|
\leq\frac{2}{k}\sum_{i=1}^{k}\left\|X_i\right\|\left\|Y_i\right\|\to 0~~~a.s.,\\
&\left\| \frac{1}{k}\sum_{i=1}^{k}X_i\right\| \left\| \frac{1}{k}\sum_{i=1}^{k}Y_i\right\|\lesssim\frac{1}{k}\sum_{i=1}^{k}L(\xi_i)\|x_{i-1}-x^*\|\to0~~~a.s.
\end{aligned}
	$$
%	Moreover, by the strong law of large numbers, $\left\|\frac{1}{k}\sum_{i=1}^{k}X_i\right\|$ is bounded, the last term on the right hand of (\ref{sigma}) tends to zero almost surely.
Then,
	$\left\|{\Sigma}_k-\Sigma\right\|\to 0$ almost surely.
The proof is complete.
\end{proof}
%
%Compared with Theorem 4.2 in
%\cite{Chen16}, {where the convergence rate {almost} $O(k^{-\frac{1}{2}})$ in expectation of the plug-in estimator  has been obtained,} Theorem \ref{T3.1} shows that the plug-in estimator converges to the true covariance matrix almost surely.
 %The underlying reason is that   we are only able to arrive at $P_{A_k}= P_A$ almost surely for sufficiently large $k$.

Next, we study the consistency of plug-in method for estimating the covariance matrix in the limit normal distribution of last iterate of SDA (Theorem \ref{thm:non-ave}).
Let $\Lambda_k, P_{A_k}, H_k, \Sigma_k, r_k$ be defined as above and  $G_k$ be the $r$-order leading {principle submatrix} of $\Lambda_k^{T} P_{A_k}  H_k P_{A_k} \Lambda_k$.
Then $\Lambda_k$ and
\bgeq
%\label{eq:sigmak}
\Sigma_{1_k}:=\int_{0}^{\infty} e^{(-G_k) t}\left(\Lambda_k^{T}\right)^{({r_k})} P_{A_k} {\Sigma_k} P_{A_k}\left( \left(\Lambda_k^{T}\right)^{({r_k})}\right) ^T e^{\left(-G_k^{T}\right) t} \mathrm{~d} t
\edeq
 are the plug-in estimators of $\Lambda$ and $ \Sigma_1$   in (\ref{Sigma1}) respectively. \footnote{We may use sample average approximation method to calculate the integration in $t$.}

\begin{thm}
	\label{T3.2}
	Suppose that (i) Assumptions \ref{assu_1}, \ref{ass-batch}  and \ref{plugass} hold, (ii) step-size
	$\alpha_{k}=\frac{\alpha_0}{k^{\beta}} \text { with }
	\beta  \in\left(\frac{2}{3}, 1\right)$ and $ \alpha_{0}>0$.
Denote $
	\begin{array}{c}
	\tilde{\Sigma}_k=\Lambda_k\left(\begin{array}{cc}
	\Sigma_{1_k} & \mathbf{0}\\
	\mathbf{0} & \mathbf{0}
	\end{array}\right) \Lambda_k^{T}.
	\end{array}
	$
 Then
	\begin{equation*}
	\left\|\tilde{\Sigma}_k-\tilde{\Sigma}\right\| \to 0~~~{a.s.,}
	\end{equation*}
	where $\tilde{\Sigma}$ is defined in (\ref{Sigma1}).
\end{thm}
\begin{proof}
By (\ref{eq:Lambda_k-Lambda})
	 and  the definitions of $\tilde{\Sigma}_k$ and $\tilde{\Sigma}$,
\begin{equation*}
\label{S-s}
	\begin{aligned}
\left\|\tilde{\Sigma}_k-\tilde{\Sigma}\right\|
&=
\left\| \Lambda_k\left(\begin{array}{cc}
\Sigma_{1_k} & \mathbf{0} \\
\mathbf{0} & \mathbf{0}
\end{array}\right) \Lambda_k^{T}-\Lambda\left(\begin{array}{cc}
\Sigma_{1} & \mathbf{0} \\
\mathbf{0} & \mathbf{0}
\end{array}\right) \Lambda^{T}\right\| \\
&\leq
\left\| \Lambda_k\left(\begin{array}{cc}
\Sigma_{1_k}-\Sigma_{1} & \mathbf{0} \\
\mathbf{0} & \mathbf{0}
\end{array}\right) \Lambda_k^{T}\right\|
%+
%\left\|  \left( \Lambda_k-\Lambda\right) \left(\begin{array}{cc}
%\Sigma_{1} & \mathbf{0} \\
%\mathbf{0} & \mathbf{0}
%\end{array}\right)\Lambda^T\right\|
+2
 \left\| \Lambda_k\right\| \left\|  \left(\begin{array}{cc}
\Sigma_{1} & \mathbf{0} \\
\mathbf{0} & \mathbf{0}
\end{array}\right)\right\| \left\| \Lambda_k^T-\Lambda^T\right\|
	\end{aligned}
\end{equation*}
for sufficiently large $k$.
Note that $\|\Lambda\|$ and $\|\Sigma_{1}\|$ are bounded,  it is sufficient to show that  $\|\Sigma_{1_k}-\Sigma_{1}\|$ converges to zero.

%As $\Lambda$ and $\Lambda_k$ are the orthogonal matrices to singular value decomposition of $P_{A_k}$ and $P_A$  respectively and $P_{A_k}\to P$ almost surely \cite[Theorem 3]{Duchi19},
% there exists a sequence $\{\Lambda_k\}$  converges to $\Lambda$. Then
% $\left\| \Lambda_k^{T}-\Lambda^T \right\|\to0$   almost surely.

{We} employ the sensitivity of  solution for Lyapunov equation \cite[Theorem 2.1]{Hewer88} to study the convergence of $\|\Sigma_{1_k}-\Sigma_{1}\|$.
Denote
$$
Q=\left(\Lambda^{T}\right)^{(r)} P_{A} {\Sigma} P_{A}\left( \left(\Lambda^{T}\right)^{(r)}\right) ^T,
$$
we have
$$
\begin{aligned}
(-G)^{T} \Sigma_1+\Sigma_1 (-G) &=(-G)^{T}\left(\int_{0}^{\infty} e^{(-G)^{T} t} Q e^{(-G) t} d t\right)+\left(\int_{0}^{\infty} e^{(-G)^{T} t} Q e^{(-G) t} d t\right) (-G) \\
&=\int_{0}^{\infty} \frac{d}{d t}\left(e^{(-G)^{T} t} Q e^{(-G) t}\right) d t \\
&=\left.e^{(-G)^{T} t} Q e^{(-G) t}\right|_{0} ^{\infty}=-Q,
\end{aligned}
$$
which means $\Sigma_{1}$   is the solution of Lyapunov equation
\begin{equation*}
\label{Lya}
(-G)^{T} X+X (-G)+Q=0,
\end{equation*}
where $G$ is defined in Theorem \ref{thm:non-ave}.
By the similar analysis,  $\Sigma_{1_k}$ is the solution of Lyapunov equation
\begin{equation*}
\label{Lya_k}
(-G_k)^{T} X+X (-G_k)+Q_k=0,
\end{equation*}
where
$$
Q_k= \left(\Lambda_k^{T}\right)^{({r_k})} P_{A_k} {\Sigma_k} P_{A_k}\left( \left(\Lambda_k^{T}\right)^{({r_k})}\right) ^T
$$
and $G_k$  is the $r_k$-order leading {principle submatrix} of $\Lambda_k^{T} P_{A_k}  H_k P_{A_k} \Lambda_k$.
%Next, we verify the conditions of the sensitivity of  solution to Lyapunov equation \cite[Theorem 2.1]{Hewer88}.
%The (\ref{Lya}) and  (\ref{Lya_k}) are exact formula (1.1) and (1.3) in \cite[Theorem 2.1]{Hewer88}.
%By the definition $-G$, the condition (1.2) of   \cite[Theorem 2.1]{Hewer88} holds.
Denote
$$
W=\int_{0}^{\infty} e^{(-G) t} e^{\left(-G^{T}\right) t} \mathrm{~d} t,$$
we have
\begin{equation*}
\begin{aligned}
\label{W}
(-G)^{T} W+W (-G)
&=
(-G)^{T} \left(\int_{0}^{\infty} e^{(-G) t} e^{\left(-G^{T}\right) t} \mathrm{~d} t\right)+\left(\int_{0}^{\infty} e^{(-G) t} e^{\left(-G^{T}\right) t} \mathrm{~d} t\right) (-G)\\
&=\int_{0}^{\infty} \frac{d}{d t}\left(e^{(-G)^{T} t} e^{(-G) t}\right) d t \\
&=\left.e^{(-G)^{T} t}  e^{(-G) t}\right|_{0} ^{\infty}=-\mathbf{I}_n.
\end{aligned}
\end{equation*}
 By  Assumption \ref{assu_1} (iii),
$$
\begin{aligned}
\left\|  W\right\|
&=
\int_{0}^{\infty}  \left\| e^{-G t}\right\|^2 \mathrm{~d} t
&\leq
\int_{0}^{\infty}  \left\| e^{-2\mu t}\right\| \mathrm{~d} t
&=\frac{1}{2\mu}.
\end{aligned}
$$
Moreover, $-G$ is stable.
Then,
by the sensitivity of  solution to Lyapunov equation \cite[Page 327, last inequality]{Hewer88},
 $$
 \begin{aligned}
 \left\| \Sigma_{1_k}-\Sigma_{1}\right\|
 &\leq
 \frac{1}{2\mu}\left[ \left\| Q_k-Q\right\|+2 \left\| G-G_k\right\|\|\Sigma_{1}\|\right].
\end{aligned}
 $$
Mimicking the proof of Theorem \ref{T3.1}, it is easy to show
 $
 \left\|Q_k-Q\right\|\to 0$ and
 $\left\|G-G_k \right\| \to 0$ almost surely. Then $\left\| \Sigma_{1_k}-\Sigma_{1}\right\|$ tends to zero almost surely. The proof is complete.
\end{proof}

%Note that $\Sigma_{1_k}$ in (\ref{eq:sigmak}) depends on the  integration,.  {It is easy to verify that  with  the sample size of SAA method tends to infinity, Theorem \ref{T3.2} holds.}

\subsection{Batch-means method}
Different with plug-in method, batch-means method only uses the iterates from {SDA} without
requiring computation of any additional quantities.
%We split $k$ iterates $\left\{x_{1}, \ldots, x_{k}\right\}$ into $M+1$ batches with sizes $n_{0}, n_{1}, \ldots, n_{M}$.
%%$$
%%\underbrace{\left\{x_{s_{0}}, \ldots, x_{e_{0}}\right\}}_{0 \text { -th batch }}, \underbrace{\left\{x_{s_{1}}, \ldots x_{e_{1}}\right\}}_{\text {1-st batch }}, \ldots, \underbrace{\left\{x_{s_{M}}, \ldots, x_{e_{M}}\right\}}_{M \text { -th batch }} .
%%$$
%Let $s_{i}$ and $e_{i}$ be the starting and ending index of $i$-th batch with $s_{0}=1$, $s_{i}=e_{i-1}+1, n_{i}=e_{i}-s_{i}+1$, $e_{M}=k$ and $E_M=e_{M}-e_0.$ Setting batch size $e_i=((i+1)N)^{\frac{1}{1-\beta}}$, where  $M=k^{\frac{1-\beta}{2}}$ and $N=\frac{k^{1-\beta}}{M+1}$.
{Let $\{x_k\}$ be a sequence of iterates of {SDA}, we define the strictly increasing
integer-valued sequence $ \{a_m\}$ with
$a_1=1$ and $a_m = \left[ Cm^\frac{2}{1-\beta}\right] _+$ for some constant $C$. Then we split the iterates into $M$ batches with the starting  index $a_m$ of $m$-th batch. {The batch-means estimator \cite[(5)]{zhu2021online} of covariance matrix in (\ref{eq:T2.1}) is given as follows:
\begin{equation}
\label{X_n}
\frac{\sum_{i=1}^{k}\left(\sum_{j=t_{i}}^{i} x_{j}-l_{i} \bar{x}_{k}\right)\left(\sum_{j=t_{i}}^{i} x_{j}-l_{i} \bar{x}_{k}\right)^{T}}{\sum_{i=1}^{k} l_{i}},
\end{equation}
where $t_i$ is determined
by the sequence $\{a_m\}$ through $t_i = a_m$ when $i\in [a_m, a_{m+1})$,  $\bar{x}_{k}=\frac{1}{k}\sum_{i=1}^{k}x_i$,  $l_{i}=i-t_{i}+1$.}}

Although the batch-means {estimator} {is the same as} the batch-means estimator for {SGD \cite{zhu2021online}}, the proof of convergence  of
(\ref{X_n})
is not straightforward at all.
If we follow \cite[Theorem 3.3]{zhu2021online}
to explore the consistency of batch-means estimator in expectation, the required convergence rate of the {iterates}  $x_{k}$ to the true solution   $x^*$     is not reachable. On the other hand, if we follow Theorems \ref{T3.1}-\ref{T3.2} to
 study {almost sure} convergence of the batch-means estimator, we are unable to show  the  convergence of  the indispensable auxiliary sequence (see the following formula (\ref{eq:U_k}))  to the true covariance matrix. Therefore, we have to establish the consistency of batch-means estimator through the techniques both for convergence in expectation and {almost sure} convergence. Following the idea of   \cite[Theorem 3.3]{zhu2021online}, we investigate  the consistency of (\ref{X_n}) by the following three steps.
\begin{enumerate}

	\setlength{\itemindent}{2em}
	\item[\textbf{{Step 1.}}]
	
	[Lemma \ref{N_nto0}] Define an auxiliary sequence $U_k$,
\begin{equation}
	\label{eq:U_k}
	U_{k}:=\left(\mathbf{I}_n-\alpha_{k-1} \mathrm{P}_{A} \nabla f\left(x^*\right) \mathrm{P}_{A} \right)U_{k-1}+\alpha_{k-1}S_{k-1},~U_0\in\mathcal{T}_{\C}\left(x^*\right),
\end{equation}
{where $S_k$ and $\mathcal{T}_{\C}\left(x^*\right)$ are defined in (\ref{de:J}) and (\ref{eq:mathcal{T}}) respectively.}
	Construct the batch-means estimator based on $U_k$
	as
	\begin{equation}
	\label{UU}
	\frac{\sum_{i=1}^{k}\left(\sum_{j=t_{i}}^{i} U_{j}-l_{i} \bar{U}_{k}\right)\left(\sum_{j=t_{i}}^{i} U_{j}-l_{i} \bar{U}_{k}\right)^{T}}{\sum_{i=1}^{k} l_{i}},
	\end{equation}
	where $\bar{U}_{k}=\frac{1}{k}\sum_{i=1}^{k}U_i$.
%	\begin{equation}
%	\label{eq:U_nk}
%	\bar{U}_{n_{i}}:=\frac{1}{n_{i}} \sum_{k=s_{i}}^{e_{i}} U_{k} \quad \text{and} \quad \bar{U}_{M}:=\frac{1}{e_{M}-e_{0}} \sum_{k=s_{1}}^{e_{M}} U_{k}.
%	\end{equation}
	{Study 
	 $$\mathbb{E}\left[\left\|  \left( {\sum_{i=1}^{k} l_{i}}\right) ^{-1}{\sum_{i=1}^{k}\left(\sum_{j=t_{i}}^{i} U_{j}-l_{i} \bar{U}_{k}\right)\left(\sum_{j=t_{i}}^{i} U_{j}-l_{i} \bar{U}_{k}\right)^{T}}-\mathrm{P}_{A}H^{\dagger} \mathrm{P}_{A} \Sigma \mathrm{P}_{A}H^{\dagger} \mathrm{P}_{A}\right\| \right]\to 0. $$}
%	 We show that $M^{-1} \sum_{i=1}^{M} n_{i}\left(\bar{U}_{n_{i}}-\bar{U}_{M}\right)\left(\bar{U}_{n_{i}}-\bar{U}_{M}\right)^{T}$ converges to $\mathrm{P}_{A}H^{\dagger} \mathrm{P}_{A} \Sigma \mathrm{P}_{A}H^{\dagger} \mathrm{P}_{A}$ in expectation.
%	The main reason for the selection of convergence expectation is that the  convergence rate in almost surely  is too large to ensure the batch-means estimator based on $U_n$ converge to $\mathrm{P}_{A}H^{\dagger} \mathrm{P}_{A} \Sigma \mathrm{P}_{A}H^{\dagger} \mathrm{P}_{A}$. Although the noise $S_k$ is martingale difference sequences, it eliminate the term which consist of the  $M^{-1} \sum_{k=1}^{M} n_{k}\left(\bar{U}_{n_{k}}-\bar{U}_{M}\right)\left(\bar{U}_{n_{k}}-\bar{U}_{M}\right)^{T}$ to guarantee the convergence.
	\item[\textbf{{Step 2.}}]
	
	[Lemma \ref{delta_kto0}] Show the difference between (\ref{X_n}) and (\ref{UU}),
{\footnotesize$$
\left\|\frac{\sum_{i=1}^{k}\left(\sum_{j=t_{i}}^{i} x_{j}-l_{i} \bar{x}_{k}\right)\left(\sum_{j=t_{i}}^{i} x_{j}-l_{i} \bar{x}_{k}\right)^{T}
-
\sum_{i=1}^{k}\left(\sum_{j=t_{i}}^{i} U_{j}-l_{i} \bar{U}_{k}\right)\left(\sum_{j=t_{i}}^{i} U_{j}-l_{i} \bar{U}_{k}\right)^{T}}{\sum_{i=1}^{k} l_{i}} \right\|\stackrel{d}{\longrightarrow}0.$$}
\item[\textbf{{Step 3.}}]

[Theorem \ref{XX}]
%The convergence in expectation and convergence in almost surely imply the convergence in distribution.
Combine the convergence in expectation {in the first step} and convergence in distribution {in the second step},
$$\left\|\frac{\sum_{i=1}^{k}\left(\sum_{j=t_{i}}^{i} x_{j}-l_{i} \bar{x}_{k}\right)\left(\sum_{j=t_{i}}^{i} x_{j}-l_{i} \bar{x}_{k}\right)^{T}}{\sum_{i=1}^{k} l_{i}}
-
P_AH^\dagger P_A\Sigma P_AH^\dagger P_A\right\|\stackrel{d}{\longrightarrow}0.$$
\end{enumerate}

%Due to the effect of constraint set, we need define the seminorm first:
%$$
%\normmm{T}_{\mathcal{T}}:=\sup \{\|T x\|: x \in \mathcal{T},\|x\| \leq 1\},
%$$
%where $\mathcal{T}$ is a subspace $\{x:Ax=0\}.$
%We take $\normmm{T}_{\mathcal{T}}=0$ if $\mathcal{T}=\{0\} .$ Clearly $\normmm{T}_{\mathcal{T}} \leq\normmm{T}_{\text {op }},$ where $\normmm{T}_{\text {op }}:=$ $\sup \{\|T x\|:\|x\| \leq 1\}$ is the $\ell_{2}$ -operator norm.
{We begin by starting some technical lemmas where the convergence of the fourth moment of  $\Delta_{k}$  and the convergence rate of $U_k$ are studied.}

\begin{lema}
	\label{Delta4-bounded} Suppose that
 (i) {Assumptions \ref{assu_1} and  \ref{ass-batch} hold},  (ii) step-size
 {$\alpha_{k} =\alpha_{0} k^{-\beta} \text { with }
 	\beta  \in\left(\frac{2}{3}, 1\right)$} and $ \alpha_{0}>0$. Let $\Delta_k$ be defined as in (\ref{de:J}). Then
	$\mathbb{E}\left[\|\Delta_{k}\|^2\right]\to 0.$
\end{lema}
\begin{proof}  By the definition of $\Delta_k$  and  Assumption \ref{ass-batch}, $\|\Delta_{k}\|^2$ is bounded.   Then the rest follows from the fact $x_k\to x^*$ \cite[Theorem 2]{Duchi19} and the Lebesgue dominated convergence theorem.
\end{proof}

\begin{lema}\emph{[Convergence rate of $U_k$]}
	\label{EU_n^2}
	Suppose that
	(i) {Assumptions \ref{assu_1} and  \ref{ass-batch} hold},  (ii) step-size
	{$\alpha_{k} =\alpha_{0} k^{-\beta} \text { with }
		\beta  \in\left(\frac{2}{3}, 1\right)$} and $ \alpha_{0}>0$. Let $U_k$ be defined as in (\ref{eq:U_k}). Then
$$
\mathbb{E}\left[\left\|U_k \right\| ^2\right]\lesssim k^{-\beta}.
$$
\end{lema}
\begin{proof}
	Define  the seminorm
	$$
	\normmm{A}_{\mathcal{T}}:=\sup \{\|A x\|: x \in \mathcal{T}_{\C}\left(x^*\right),\|x\| \leq 1\},
	$$
	where $\normmm{A}_{\mathcal{T}}=0$ if $\mathcal{T}=\{0\} .$  By the definition (\ref{eq:U_k}), $U_k\in \mathcal{T}_{\C}\left(x^*\right), \forall k\geq 0$. {Recall the filtration $\mathcal{F}_{k}$ defined in (\ref{eq:filtration}). Then, there exists a constant $C$ such that}
	$$
	\begin{aligned}
\mathbb{E}\left[ \left\|U_k \right\| ^2|\mathcal{F}_{k-1}\right]
=&
\mathbb{E}\left[
\left\| \left(\mathbf{I}_n-\alpha_{k-1} \mathrm{P}_{A} \nabla f\left(x^*\right) \mathrm{P}_{A} \right)U_{k-1}+\alpha_{k-1}S_{k-1}\right\|
^2|\mathcal{F}_{k-1}
\right]\\
\leq& \normmm{\mathbf{I}_n-\alpha_{k-1} \nabla f\left(x^*\right)}^2_{\mathcal{T}}\left\| U_{k-1}\right\|
 ^2+\alpha_{k-1}^2\left(\|\Sigma\|+4L^2\|\Delta_k\|^2+4L\|\Sigma\|^{\frac{1}{2}}\|\Delta_k\|\right)\\
\leq& \left( 1-\mu\alpha_{k-1} \right)\left\|U_{k-1} \right\| ^2+C\alpha_{k-1}^2,
	\end{aligned}
	$$
	where the first inequality follows from  (\ref{eq:S_k,1}), the second follows from Assumptions \ref{assu_1} (iii) and  \ref{ass-batch}.
%Then for $k>\min\{k:\alpha_{k-1}\lambda_{\min}(H)\geq \frac{1}{2}\mu\alpha_{k-1}\}$,
%$$
%\mathbb{E}\left[ \left\|U_k \right\| ^2|\mathcal{F}_{k-1}\right]\leq \left( 1-\frac{1}{2}\mu\alpha_{k-1} \right)\left\|U_{k-1} \right\| ^2+C\alpha_{k-1}^2.
%$$
Then the rest of proof is same as the proof of \cite[Lemma B.3]{Chen16}.
\end{proof}

\begin{lema}\emph{[Convergence rate of $\rho_k$]}
	\label{rho_na.s.}
	Suppose that (i) Assumption \ref{assu_1} hold, (ii) step-size
	{$\alpha_{k} =\alpha_{0} k^{-\beta} \text { with }
		\beta  \in\left(\frac{2}{3}, 1\right)$} and $ \alpha_{0}>0$.
Denote
\begin{equation}
		\label{delta_k}
	\rho_{k}:=\left(\mathbf{I}_n-\alpha_{k-1} P_AHP_A\right) \rho_{k-1}+\alpha_{k-1}\left( \zeta_{k-1}+\epsilon_{k-1}\right),
	\end{equation}
	{where $\rho_0=0_n$, $\zeta_{k-1}$ and $\epsilon_{k-1}$ are defined in (\ref{de:J}).}
 Then for any $\delta\in(0,1-\frac{1}{2\beta})$, $\gamma\in(0,2\delta+1-\frac{1}{\beta})$,
	\begin{equation*}
	\label{gamma_rate}
	\|\rho_k\|=o(\alpha_{k}^\gamma)~~~a.s.
	\end{equation*}
\end{lema}
\begin{proof} The proof is similar to Theorem \ref{xn-x*}.
\end{proof}
{We are ready for the Step 1.}

\begin{lema}
	\label{N_nto0}
Suppose that (i) Assumptions \ref{assu_1}, \ref{ass-batch} and \ref{plugass} hold, (ii)
step-size
{$\alpha_{k} =\alpha_{0} k^{-\beta} \text { with }
	\beta  \in\left(\frac{2}{3}, 1\right)$} and $ \alpha_{0}>0$, {(iii) $a_m = [Cm^\tau]_+$, where $C > 0$ and $\tau > 1/(1-\beta)$.} Then,
\begin{equation*}
\label{U-Sigma}
	\mathbb{E}\left[\left\|  \left( {\sum_{i=1}^{k} l_{i}}\right) ^{-1}{\sum_{i=1}^{k}\left(\sum_{j=t_{i}}^{i} U_{j}-l_{i} \bar{U}_{k}\right)\left(\sum_{j=t_{i}}^{i} U_{j}-l_{i} \bar{U}_{k}\right)^{T}}-\mathrm{P}_{A}H^{\dagger} \mathrm{P}_{A} \Sigma \mathrm{P}_{A}H^{\dagger} \mathrm{P}_{A}\right\| \right]\to 0
\end{equation*}
$as~k\to\infty.$
\end{lema}
\begin{proof}
	By the triangle inequality,
%	$$
%\begin{aligned}
%	&M^{-1} \sum_{i=1}^{M} n_{i}\left(\bar{U}_{n_{i}}-\bar{U}_{M}\right)\left(\bar{U}_{n_{i}}-\bar{U}_{M}\right)^{T}
%	=M^{-1} \sum_{i=1}^{M} n_{i} \bar{U}_{n_{i}} \bar{U}_{n_{i}}^{T}-M^{-1} \sum_{i=1}^{M} n_{i} \bar{U}_{M} \bar{U}_{M}^{T}.
%\end{aligned}
%	$$
	\begin{equation}
	\label{3.50}
	\begin{aligned}
	&\mathbb{E}\left[ \left\|
	\left( {\sum_{i=1}^{k} l_{i}}\right) ^{-1}{\sum_{i=1}^{k}\left(\sum_{j=t_{i}}^{i} U_{j}-l_{i} \bar{U}_{k}\right)\left(\sum_{j=t_{i}}^{i} U_{j}-l_{i} \bar{U}_{k}\right)^{T}}
	-P_AH^\dagger P_A\Sigma P_AH^\dagger P_A\right\|\right]\\
	&\leq\mathbb{E}\left[ \left\|
	\left( {\sum_{i=1}^{k} l_{i}}\right) ^{-1}{\sum_{i=1}^{k}\left(\sum_{j=t_{i}}^{i} U_{j}\right)\left(\sum_{j=t_{i}}^{i} U_{j}\right)^{T}}
	-P_AH^\dagger P_A\Sigma P_AH^\dagger P_A\right\|\right]\\
	&+
	\mathbb{E}\left\|\left(\sum_{i=1}^{k} l_{i}\right)^{-1} \sum_{i=1}^{k} l_{i}^{2} \bar{U}_{k} \bar{U}_{k}^{T}\right\|
	+2 \mathbb{E}\left\|\left(\sum_{i=1}^{k} l_{i}\right)^{-1} \sum_{i=1}^{k}\left(\sum_{j=t_{i}}^{i} U_{j}\right)\left(l_{i} \bar{U}_{k}\right)^{T}\right\|.
	\end{aligned}
	\end{equation}
	Then, we may finish the proof by studying the convergence of
	 the three terms on the right hand of (\ref{3.50}).
	
	We first focus on the first term on the right hand of (\ref{3.50}).
	Denote the following matrices sequences,
\begin{equation}
\label{Yjk}
	Y_{p}^{k}=\prod_{i=p}^{k-1}\left(\mathbf{I}_n-\alpha_{i} P_A\nabla f(x^*)P_A\right), \quad Y_{i}^{i}=\mathbf{I}_n, \quad \inmat{for}\;\;k >p,
\end{equation}
	the recursion of $U_{k}$ (\ref{eq:U_k}) can be rewritten as
	$$
	\begin{aligned}
	U_{k}=Y_{t_i-1}^{k} U_{t_i-1}+\sum_{p=t_i}^{k} Y_{p}^{k} \alpha_{p-1}S_{p-1}, \quad \inmat{for}\;\;k\in[t_i,i],
	\end{aligned}
	$$
	where $S_{p-1}$ is defined in (\ref{de:J}).
%	where $\sigma_{i}= \mathrm{P}_{A} \xi_{i}+ \mathrm{P}_{A} \zeta_{i}-\alpha_{i}^{-1}\varepsilon_{i}$.
	Then we have
\begin{equation*}
	\label{bar{U}_{n_{i}}}
		\begin{aligned}
			&\left(\sum_{i=1}^{k} l_{i}\right)^{-1} \sum_{i=1}^{k}\left(\sum_{j=t_{i}}^{i} U_{j}\right)\left(\sum_{j=t_{i}}^{i} U_{j}\right)^{T} \\
			=&\left(\sum_{i=1}^{k} l_{i}\right)^{-1} \sum_{i=1}^{k}\left(S_{t_{i}-1}^{i} U_{t_{i}-1}+\sum_{p=t_{i}}^{i}\left(\mathbf{I}_n+S_{p}^{i}\right) \alpha_{p-1} S_{p-1}\right)\left(S_{t_{i}-1}^{i} U_{t_{i}-1}+\sum_{p=t_{i}}^{i}\left(\mathbf{I}_n+S_{p}^{i}\right) \alpha_{p-1} S_{p-1}\right)^{T} \\
			=&\left(\sum_{i=1}^{k} l_{i}\right)^{-1} \sum_{i=1}^{k}\left(P_AH^\dagger P_A\left(\sum_{p=t_{i}}^{i} S_{p-1}\right)\left(\sum_{p=t_{i}}^{i} S_{p-1}\right)^{T} P_AH^\dagger P_A+\Phi_i \Upsilon_i^{T}+\Upsilon_i \Phi_i^{T}+\Phi_i \Phi_i^{T}\right),
		\end{aligned}
\end{equation*}
where
\begin{equation}
\label{Sme}
\begin{cases}
	\Upsilon_i:={P_AH^\dagger P_A \sum_{p=t_{i}}^{{i}} S_{p-1}},\\
	 \Phi_i:=S_{t_{i}-1}^{i} U_{t_{i}-1}+\sum_{p=t_{i}}^{i}\left(\alpha_{p-1} S_{p}^{i}+\alpha_{p-1} \mathbf{I}_n-P_AH^\dagger P_A\right) S_{p-1}, \\
	S_{p}^{i}:=\sum_{l=p+1}^{i} Y_{p}^{l}=\sum_{l=p}^{i} Y_{j}^{l}-\mathbf{I}_n.
\end{cases}
\end{equation}
	Subsequently, the first term on the right hand of (\ref{3.50})
	\begin{equation}
	\label{3.53}
	\begin{aligned}
		& \mathbb{E}\left\|\left(\sum_{i=1}^{k} l_{i}\right)^{-1} \sum_{i=1}^{k}\left(\sum_{j=t_{i}}^{i} U_{j}\right)\left(\sum_{j=t_{i}}^{i} U_{j}\right)^{T}-P_AH^\dagger P_A\Sigma P_AH^\dagger P_A\right\| \\
		\leq & \left\|P_AH^\dagger P_A\right\|^2
		E\left\|\underbrace{\left(\sum_{i=1}^{k} l_{i}\right)^{-1} \sum_{i=1}^{k}\left(\sum_{p=t_{i}}^{i} S_{p-1}\right)\left(\sum_{p=t_{i}}^{i} S_{p-1}\right)^{T}}_{I_1} -\Sigma \right\| \\
		&+\mathbb{E}\left\|\left(\sum_{i=1}^{k} l_{i}\right)^{-1} \sum_{i=1}^{k} \Phi_{i} \Phi_{i}^{T}\right\|
		+
		2\mathbb{E}\left\|\left(\sum_{i=1}^{k} l_{i}\right)^{-1} \sum_{i=1}^{k} \Phi_{i} \Upsilon_{i}^{T}\right\|.
	\end{aligned}
	\end{equation}
%	Then, it is sufficient to show
%$
%	\label{SS-Sigma}
%	\mathbb{E}\left[ \left\| I_1-\Sigma\right\|\right]\to 0
%$, which implies the first term on the right hand of (\ref{3.53}) tends to zero.

Next, we mimic the proof of \cite[Lemma B.2.]{zhu2021online} to show $\mathbb{E}\left[ \left\|  I_1-\Sigma\right\| \right]\to0$,  which implies the first term on the right hand of (\ref{3.53}) tends to zero.
Denote $\widetilde{S}_p=P_A[F(x^*,\xi_p)-f(x^*)]$ and
$I_2=\left(\sum_{i=1}^{k} l_{i}\right)^{-1} \sum_{i=1}^{k}\left(\sum_{p=t_{i}}^{i} \widetilde{S}_{p-1}\right)\left(\sum_{p=t_{i}}^{i} \widetilde{S}_{p-1}\right)^{T}$, we have
\begin{equation}
\label{s-S}
{\begin{aligned}
	&\mathbb{E}\left[ \left\|  I_1-\Sigma\right\| \right]
\leq&
\mathbb{E}\left[ \left\|  I_2-\Sigma\right\| \right]
+
\mathbb{E}\left[ \left\|  I_1-I_2\right\| \right].
\end{aligned}}
	\end{equation}
By the definition of $\Sigma$ and the fact $\{\widetilde{S}_{p}\}$ is iid, $\mathbb{E}(I_2)=\left(\sum_{i=1}^{k} l_{i}\right)^{-1} \sum_{i=1}^{k}\sum_{p=t_{i}}^{i}\mathbb{E} \widetilde{S}_{p-1} \widetilde{S}_{p-1}^{T}=\Sigma$. $\mathbb{E}\left[  I_2 \right]^2$ can be expanded into two parts,
\begin{equation*}
	\begin{aligned}
		\mathbb{E}\left[  I_2 \right]^2 &=\mathbb{E}\left(\sum_{i=1}^{k} l_{i}\right)^{-2} \sum_{1 \leq i, j \leq k}\left(\sum_{p=t_{i}}^{i} \tilde{S}_{p}\right)\left(\sum_{p=t_{i}}^{i} \tilde{S}_{p}\right)^{T}\left(\sum_{p=t_{j}}^{j} \tilde{S}_{p}\right)\left(\sum_{p=t_{j}}^{j} \tilde{S}_{p}\right)^{T} \\
		&=\left(\sum_{i=1}^{k} l_{i}\right)^{-2} I_3+\left(\sum_{i=1}^{k} l_{i}\right)^{-2} I_4,
	\end{aligned}
\end{equation*}
where
$$
\footnotesize
\begin{aligned}
	I_3 &=\mathbb{E} \sum_{m=1}^{M-1} \sum_{i=a_{m}}^{a_{m+1}-1}\left[2 \sum_{j=a_{m}}^{i-1} \sum_{a_{m} \leq p_{1} \neq p_{2} \leq j}\left(\tilde{S}_{p_{1}} \tilde{S}_{p_{2}}^{ T} \tilde{S}_{p_{1}}^{} \tilde{S}_{p_{2}}^{ T}+\tilde{S}_{p_{1}}^{} \tilde{S}_{p_{2}}^{ T} \tilde{S}_{p_{2}}^{} \tilde{S}_{p_{1}}^{ T}\right)+\sum_{a_{m} \leq p_{1} \neq p_{2} \leq i}\left(\tilde{S}_{p_{1}}^{} \tilde{S}_{p_{2}}^{ T} \tilde{S}_{p_{1}}^{} \tilde{S}_{p_{2}}^{ T}+\tilde{S}_{p_{1}}^{} \tilde{S}_{p_{2}}^{ T} \tilde{S}_{p_{2}}^{} \tilde{S}_{p_{1}}^{ T}\right)\right] \\
	&+\mathbb{E} \sum_{i=a_{M}}^{k}\left[2 \sum_{j=a_{M}}^{i-1} \sum_{a_{M} \leq p_{1} \neq p_{2} \leq j}\left(\tilde{S}_{p_{1}}^{} \tilde{S}_{p_{2}}^{ T} \tilde{S}_{p_{1}}^{} \tilde{S}_{p_{2}}^{ T}+\tilde{S}_{p_{1}}^{} \tilde{S}_{p_{2}}^{ T} \tilde{S}_{p_{2}}^{} \tilde{S}_{p_{1}}^{ T}\right)+\sum_{a_{M} \leq p_{1} \neq p_{2} \leq i}\left(\tilde{S}_{p_{1}}^{} \tilde{S}_{p_{2}}^{ T} \tilde{S}_{p_{1}}^{} \tilde{S}_{p_{2}}^{ T}+\tilde{S}_{p_{1}}^{} \tilde{S}_{p_{2}}^{ T} \tilde{S}_{p_{2}}^{} \tilde{S}_{p_{1}}^{ T}\right)\right]
\end{aligned}
$$
and
\begin{equation*}
	I_4=\sum_{i=1}^{k} \sum_{j=1}^{k} \sum_{p=t_{i}}^{i} \sum_{q=t_{j}}^{j} \mathbb{E}\left(\tilde{S}_{p}^{} \tilde{S}_{p}^{ T} \tilde{S}_{q}^{} \tilde{S}_{q}^{ T}\right).
\end{equation*}
Then, the first term on the right hand of (\ref{s-S})
\begin{equation}
	\label{3.51}
	\begin{aligned}
		\mathbb{E}\left[ \left\|  I_2-\Sigma\right\| \right]
		\leq&
		\sqrt{ \left\|\mathbb{E}\left[  I_2 \right]^2-\Sigma^2\right\|} \leq
		\sqrt{\left\|\left(\sum_{i=1}^{k} l_{i}\right)^{-2}I_4-\Sigma^2 \right\| +\left(\sum_{i=1}^{k} l_{i}\right)^{-2}\|I_3\|}.
	\end{aligned}
\end{equation}
We first focus on the first term on the right hand of (\ref{3.51}). Consider two cases, one is when $p$ and $q$ are in the same block,  
$$
I_5=\sum_{m=1}^{M} \sum_{i=a_{m}}^{a_{m+1}-1} \sum_{j=a_{m}}^{a_{m+1}-1} \sum_{p=a_{m}}^{i} \sum_{q=a_{m}}^{j}\left\|\mathbb{E}\left(\tilde{S}_{p}^{} \tilde{S}_{p}^{ T} \tilde{S}_{q}^{} \tilde{S}_{q}^{ T}\right)-\Sigma^{2}\right\|
$$
and the other is when $p$ and $q$ are in different blocks,
$$
I_6=\sum_{m \neq k} \sum_{j=a_{k}}^{a_{k+1}-1} \sum_{i=a_{m}}^{a_{m+1}-1} \sum_{q=a_{k}}^{j} \sum_{p=a_{m}}^{i}\left\|\mathbb{E}\left(\tilde{S}_{p}^{} \tilde{S}_{p}^{ T} \tilde{S}_{q}^{} \tilde{S}_{q}^{ T}\right)-\Sigma^{2}\right\|.
$$
Then, we have
\begin{equation*}
	\label{eq:II}
	\left\|\left(\sum_{i=1}^{k} l_{i}\right)^{-2} I_4-\Sigma^{2}\right\| \lesssim\left(\sum_{i=1}^{a_{M+1}-1} l_{i}\right)^{-2} I_5+\left(\sum_{i=1}^{a_{M+1}-1} l_{i}\right)^{-2} I_6.
\end{equation*}
Under Assumption \ref{plugass} (ii), $\left\|\mathbb{E}\left(\tilde{S}_{p}^{} \tilde{S}_{p}^{ T} \tilde{S}_{q}^{} \tilde{S}_{q}^{ T}\right)\right\|$ is bounded by constant $C$. 
Following \cite[(48)]{zhu2021online}, 
\begin{equation*}
	\label{eq:III}
\begin{aligned}
	\left(\sum_{i=1}^{a_{M+1}-1} l_{i}\right)^{-2} I_5 & \leq\left(\sum_{i=1}^{a_{M+1}-1} l_{i}\right)^{-2} \sum_{m=1}^{M} \sum_{i=a_{m}}^{a_{m+1}-1} \sum_{j=a_{m}}^{a_{m+1}-1} \sum_{p=a_{m}}^{i} \sum_{q=a_{m}}^{j}\left(C+\left\|\Sigma^{2}\right\|\right) \\
	& \lesssim\left(\sum_{i=1}^{a_{M+1}-1} l_{i}\right)^{-2} \sum_{m=1}^{M}\left(\sum_{i=a_{m}}^{a_{m+1}-1} l_{i}\right)^{2} \to 0.
\end{aligned}
\end{equation*}
The fact $\mathbb{E}[\tilde{S}_p\tilde{S}_p^T]=\Sigma$ implies  $\left(\sum_{i=1}^{a_{M+1}-1} l_{i}\right)^{-2} I_6=0$. Then, the first term on the right hand of (\ref{3.51}) tends to zero.
Based on Assumption \ref{plugass} (ii), $\left\|\mathbb{E}\left(\tilde{S}_{p_1}\tilde{S}_{p_2}^T\tilde{S}_{p_3} \tilde{S}_{p_4}^T\right)  \right\|$ is still bounded by constant $C$ for any $p_r$, $r\in\{1,2,3,4\}$.
By \cite[(45-46)]{zhu2021online},
the second term on the right hand of (\ref{3.51}) 
	{\small$$\left(\sum_{i=1}^{k} l_{i}\right)^{-2}\|I_3\| 
	\leq 
	\left(\sum_{i=1}^{k} l_{i}\right)^{-2}\sum_{m=1}^{M} \sum_{i=a_{m}}^{a_{m+1}-1}\left[2 \sum_{j=a_{m}}^{i-1} \sum_{a_{m} \leq p_{1} \neq p_{2} \leq j}(C+C)+\sum_{a_{m} \leq p_{1} \neq p_{2} \leq i}(C+C)\right]
	\to 0.$$}

Next, we study the convergence of the second term on the right hand of (\ref{s-S}).
Denote  $\bar{S}_j=-S_j-\widetilde{S}_j$, we have
\begin{equation}
\label{I_1-I_2}
\small
\begin{gathered}
	\mathbb{E}\left\|I_1-I_2\right\|
	=
	\mathbb{E}\left\|\left(\sum_{i=1}^{k} l_{i}\right)^{-1} \sum_{i=1}^{k}\left[\left(\sum_{j=t_{i}}^{i} S_{j-1}\right)\left(\sum_{j=t_{i}}^{i} S_{j-1}\right)^{T}-\left(\sum_{j=t_{i}}^{i} \tilde{S}_{j-1}\right)\left(\sum_{j=t_{i}}^{i} \tilde{S}_{j-1}\right)^{T}\right]\right\| \\
	\leq 
	2 \mathbb{E}\left\|\left(\sum_{i=1}^{k} l_{i}\right)^{-1} \sum_{i=1}^{k}\left(\sum_{j=t_{i}}^{i} \bar{S}_{j-1}\right)\left(\sum_{j=t_{i}}^{i} \tilde{S}_{j-1}\right)^{T}\right\|+\mathbb{E}\left\|\left(\sum_{i=1}^{k} l_{i}\right)^{-1} \sum_{i=1}^{k}\left(\sum_{j=t_{i}}^{i} \bar{S}_{j-1}\right)\left(\sum_{j=t_{i}}^{i} \bar{S}_{j-1}\right)^{T}\right\|.
\end{gathered}
\end{equation}
Apply Cauchy's inequality
{$$
\footnotesize
\begin{aligned}
	&\mathbb{E}\left\|\left(\sum_{i=1}^{k} l_{i}\right)^{-1} \sum_{i=1}^{k}\left(\sum_{j=t_{i}}^{i} \bar{S}_{j-1}\right)\left(\sum_{j=t_{i}}^{i} \tilde{S}_{j-1}\right)^{T}\right\| 
	\leq \sqrt{\mathbb{E}\left\|I_2\right\|} \sqrt{\mathbb{E}\left\|\left(\sum_{i=1}^{k} l_{i}\right)^{-1} \sum_{i=1}^{k}\left(\sum_{j=t_{i}}^{i} \bar{S}_{j-1}\right)\left(\sum_{j=t_{i}}^{i} \bar{S}_{j-1}\right)^{T}\right\|}.
\end{aligned}
$$}
Note that $\{\bar{S}_j\}$ is a martingale difference sequence,
\begin{equation}
	\label{eq:barS}
	\begin{aligned}
		 \mathbb{E}\left\|\left(\sum_{i=1}^{k} l_{i}\right)^{-1} \sum_{i=1}^{k}\left(\sum_{j=t_{i}}^{i} \bar{S}_{j-1}\right)\left(\sum_{j=t_{i}}^{i} \bar{S}_{j-1}\right)^{T}\right\| \leq&\left(\sum_{i=1}^{k} l_{i}\right)^{-1} \sum_{i=1}^{k} \mathbb{E}\left\|\sum_{j=t_{i}}^{i} \bar{S}_{j-1}\right\|^{2} \\
		= &\left(\sum_{i=1}^{k} l_{i}\right)^{-1} \sum_{i=1}^{k} \sum_{j=t_{i}}^{i} \mathbb{E}\left\|\bar{S}_{j-1}\right\|^{2}. 
	\end{aligned}
\end{equation}
Following Assumption \ref{assu_1} (i),
\begin{equation*}
\mathbb{E}\left\|\bar{S}_{j-1}\right\|^{2}\leq 4L^2\|\Delta_{j-1}\|^2,
\end{equation*}	
then (\ref{eq:barS}) tends to zero by Lemma \ref{Delta4-bounded}.
Combining (\ref{3.51}) and (\ref{I_1-I_2}),
the first term on the right hand of (\ref{3.53}) tends to zero.

On the other hand, by mimicking the analysis
on \cite[(63)-(67)]{zhu2021online} with Lemma \ref{EU_n^2}, the second term on the right hand of (\ref{3.53})
%\begin{equation*}
%	M^{-1} \sum_{k=1}^{M} n_{k}^{-1} \mathbb{E}\left[ \left\|\Phi_k\right\|_{2}^{2}\right]  \lesssim  N^{-1},
%\end{equation*}
%which implies
\begin{equation*}
\label{3.55}
	\mathbb{E}\left\|\left(\sum_{i=1}^{k} l_{i}\right)^{-1} \sum_{i=1}^{k} \Phi_{i} \Phi_{i}^{T}\right\| \to 0.
\end{equation*}
Using Cauchy's inequality,
$$
\mathbb{E}\left\|\left(\sum_{i=1}^{k} l_{i}\right)^{-1} \sum_{i=1}^{k} \Phi_{i} \Upsilon_{i}^{T}\right\|\leq
\sqrt{\mathbb{E}\left\|\left(\sum_{i=1}^{k} l_{i}\right)^{-1} \sum_{i=1}^{k} \Phi_{i} \Phi_{i}^{T}\right\|
\mathbb{E}\left\|\left(\sum_{i=1}^{k} l_{i}\right)^{-1} \sum_{i=1}^{k} \Upsilon_{i} \Upsilon_{i}^{T}\right\|}.
$$
Combining the fact $\mathbb{E}\left\|\left(\sum_{i=1}^{k} l_{i}\right)^{-1} \sum_{i=1}^{k} \Upsilon_{i} \Upsilon_{i}^{T}\right\| $ is bounded,
Slutsky's Theorem implies the last term on the right hand of (\ref{3.53}) tends to zero.
Summarizing above, the first term on the right hand of (\ref{3.50}) converges to zero.

 Next, we focus on second term on the right hand of (\ref{3.50}).
Note that $\mathbb{E} \left\|\bar{U}_{k} \bar{U}_{k}^{T}\right\|\leq k^{-2} \operatorname{tr}\mathbb{E}\left[ \left(\sum_{i=1}^{k}U_i \right)\left(\sum_{i=1}^{k}U_i \right)^T \right] $,
%	$$
%	\bar{U}_{M}=K^{-1} S_{s_{1}-1}^{e_{M}} U_{s_{1}-1}-K^{-1} \sum_{i=s_{1}}^{e_{M}}\left(S_{i}^{e_{M}}+I\right) \alpha_{i-1}P_AS_{i-1}.
%	$$
then
	\begin{equation}
	\label{3.56}\footnotesize
	{\begin{aligned}
	&\mathbb{E}\left\|\left(\sum_{i=1}^{k} l_{i}\right)^{-1} \sum_{i=1}^{k} l_{i}^{2} \bar{U}_{k} \bar{U}_{k}^{T}\right\|\\
	&\leq\left(\sum_{i=1}^{k} l_{i}\right)^{-1} \sum_{i=1}^{k} l_{i}^{2}\mathbb{E} \left\|\bar{U}_{k} \bar{U}_{k}^{T}\right\|\\
&\leq k^{-2} \left(\sum_{i=1}^{k} l_{i}\right)^{-1} \sum_{i=1}^{k} l_{i}^{2}\left(\|S^k_0\|^2\mathbb{E}\|U_0\|^2+\sum_{p=1}^{k}\|(\mathbf{I}_n+S^k_p)\|^2\alpha_{p-1}^2\left(\|\Sigma\|+4L^2\mathbb{E}\left[\|\Delta_{p-1}\|^2\right]+4L\|\Sigma\|^{\frac{1}{2}}\mathbb{E}\left[\|\Delta_{p-1}\|\right] \right)  \right).\\
	\end{aligned}}
	\end{equation}
\cite[(76)-(77)]{zhu2021online} and Lemma \ref{Delta4-bounded} imply  (\ref{3.56}) tends to zero. 
%	$$
%	M^{-1} K^{-1}\left\|S_{s_{1}-1}^{e_{M}}\right\|^{2} \mathbb{E} \left[ \left\|U_{s_{1}-1}\right\|_{2}^{2}\right] \to 0~~~~as~K\to\infty.
%	$$
Then, we only need the last term on the right hand of (\ref{3.50}) tends to zero. 
By Cauchy's inequality,
\begin{equation}
	\label{eq:last U}
	\begin{aligned}
	& \mathbb{E}\left\|\left(\sum_{i=1}^{k} l_{i}\right)^{-1} \sum_{i=1}^{k}\left(\sum_{j=t_{i}}^{i} U_{j}\right)\left(l_{i} \bar{U}_{k}\right)^{T}\right\| \\
	\leq & \sqrt{\frac{\mathbb{E}\left\|\sum_{i=1}^{k}\left(\sum_{j=t_{i}}^{i} U_{j}\right)\left(\sum_{j=t_{i}}^{i} U_{j}\right)^{T}\right\|}{\sum_{i=1}^{k} l_{i}} \frac{\mathbb{E}\left\|\sum_{i=1}^{k} l_{i}^{2} \bar{U}_{k} \bar{U}_{k}^{T}\right\|}{\sum_{i=1}^{k} l_{i}}} .
\end{aligned}
\end{equation}
%	Therefore,
%	$$
%	{M}^{-1} 	K^{-1} \sum_{i=s_{1}}^{e_{M}}\alpha_{i-1}^2\left\|S_{i}^{e_{M}}+I\right\|^{2}\left(\operatorname{tr}(\Sigma)+C\mathbb{E}\left[\|\Delta_{i-1}\|+\|\Delta_{i-1}\|^2\right]\right)\to 0~~~~as~K\to\infty.
%	$$
The left term of (\ref{eq:last U}) is bound by (\ref{3.53}), Slutsky's Theorem implies the last term on the right hand of (\ref{3.50}) tends to zero.
 The proof is complete.
\end{proof}
Next, we move to Step 2.
\begin{lema}
	\label{delta_kto0}
	Suppose that (i) Assumption \ref{assu_1} holds, (ii)
step-size
	{$\alpha_{k} =\alpha_{0} k^{-\beta} \text { with }
	\beta  \in\left(\frac{7}{9}, 1\right)$ and $ \alpha_{0}>0$}, (iii) $a_m = [Cm^\tau]_+$, where $C > 0$ and $\tau > 1/(1-\beta)$. Then,
	$${\small\left\|\frac{\sum_{i=1}^{k}\left(\sum_{j=t_{i}}^{i} x_{j}-l_{i} \bar{x}_{k}\right)\left(\sum_{j=t_{i}}^{i} x_{j}-l_{i} \bar{x}_{k}\right)^{T}
			-
			\sum_{i=1}^{k}\left(\sum_{j=t_{i}}^{i} U_{j}-l_{i} \bar{U}_{k}\right)\left(\sum_{j=t_{i}}^{i} U_{j}-l_{i} \bar{U}_{k}\right)^{T}}{\sum_{i=1}^{k} l_{i}} \right\|\stackrel{d}{\longrightarrow}0}$$
		$\text{as} ~k\to\infty.$
\end{lema}
\begin{proof}
By the definition of $\rho_{k}$ and $\bar{\rho}_{k}=\frac{1}{k}\sum_{i=1}^{k}\rho_i$,
			\begin{equation}
				\footnotesize
			\label{dec}
			\begin{aligned}
			&\left\|\frac{\sum_{i=1}^{k}\left(\sum_{j=t_{i}}^{i} x_{j}-l_{i} \bar{x}_{k}\right)\left(\sum_{j=t_{i}}^{i} x_{j}-l_{i} \bar{x}_{k}\right)^{T}
				-
				\sum_{i=1}^{k}\left(\sum_{j=t_{i}}^{i} U_{j}-l_{i} \bar{U}_{k}\right)\left(\sum_{j=t_{i}}^{i} U_{j}-l_{i} \bar{U}_{k}\right)^{T}}{\sum_{i=1}^{k} l_{i}} \right\|\\
			&\leq 2  \left\|\left(\sum_{i=1}^{k} l_{i}\right)^{-1} \sum_{i=1}^{k}\left(\sum_{j=t_{i}}^{i} U_{j}-l_{i} \bar{U}_{k}\right)\left(\sum_{j=t_{i}}^{i} \rho_{j}-l_{i} \bar{\rho}_{k}\right)^{T}\right\| 
			+ \left\|\left(\sum_{i=1}^{k} l_{i}\right)^{-1} \sum_{i=1}^{k}\left(\sum_{j=t_{i}}^{i} \rho_{j}-l_{i} \bar{\rho}_{k}\right)\left(\sum_{j=t_{i}}^{i} \rho_{j}-l_{i} \bar{\rho}_{k}\right)^{T}\right\|,
			\end{aligned}
			\end{equation}
where the inequality follows from  Young's inequality.

Using Cauchy’s inequality, we have
\begin{equation}
	\footnotesize
	\label{eq:Ukrho}
	\begin{aligned}
		&  \left\|\left(\sum_{i=1}^{k} l_{i}\right)^{-1} \sum_{i=1}^{k}\left(\sum_{j=t_{i}}^{i} U_{j}-l_{i} \bar{U}_{k}\right)\left(\sum_{j=t_{i}}^{i} \rho_{j}-l_{i} \bar{\rho}_{k}\right)^{T}\right\|\\ 
	&\leq \sqrt{\left\|\left(\sum_{i=1}^{k} l_{i}\right)^{-1} \sum_{i=1}^{k}\left(\sum_{j=t_{i}}^{i} U_{j}-l_{i} \bar{U}_{k}\right)\left(\sum_{j=t_{i}}^{i} U_{j}-l_{i} \bar{U}_{k}\right)^{T}\right\|
	\left\|\left(\sum_{i=1}^{k} l_{i}\right)^{-1} \sum_{i=1}^{k}\left(\sum_{j=t_{i}}^{i} \rho_{j}-l_{i} \bar{\rho}_{k}\right)\left(\sum_{j=t_{i}}^{i} \rho_{j}-l_{i} \bar{\rho}_{k}\right)^{T}\right\|}.
	\end{aligned}
\end{equation}
Claim that the left term of (\ref{eq:Ukrho}) is bounded by Lemma \ref{N_nto0},  
we only need to show the second term on the right hand of (\ref{dec}) tends to zero. By triangle inequality,
\begin{equation}
	\begin{aligned}
		\label{eq:rho}
		&\left\|\left(\sum_{i=1}^{k} l_{i}\right)^{-1} \sum_{i=1}^{k}\left(\sum_{j=t_{i}}^{i} \rho_{j}-l_{i} \bar{\rho}_{k}\right)\left(\sum_{j=t_{i}}^{i} \rho_{j}-l_{i} \bar{\rho}_{k}\right)^{T}\right\|\\
		\lesssim&\left(\sum_{i=1}^{k} l_{i}\right)^{-1} \sum_{i=1}^{k} \left\|\sum_{j=t_{i}}^{i} \rho_{j}\right\|^{2}+\left(\sum_{i=1}^{k} l_{i}\right)^{-1} \sum_{i=1}^{k} l_{i}^{2} \left\|\bar{\rho}_{k}\right\|^{2}.
	\end{aligned}
\end{equation}
Next, we focus on the first term on the right hand of (\ref{eq:rho}).
By the definition of $\rho_{k}$, $Y_{j}^{k}$ and $S_{j}^{k}$ in (\ref{delta_k}), (\ref{Yjk}) and (\ref{Sme}) respectively,
	$$
	\begin{aligned}
	\rho_{k} &=\left(\mathbf{I}_n-\alpha_{k-1} P_AHP_A\right) \rho_{k-1}+\alpha_{k-1}\left(\zeta_{k-1}+\epsilon_{k-1}\right) \\
	&=Y_{t_{i}-1}^{k} \rho_{t_{i}-1}+\sum_{p=t_{i}}^{k} Y_{p}^{k} \alpha_{p-1}\left(\zeta_{p-1}+\epsilon_{p-1}\right)\\
	&=\sum_{p=1}^{k}Y_{p}^{k}\alpha_{p-1}\left(\zeta_{p-1}+\epsilon_{p-1}\right).
\end{aligned}
$$
Then,
\begin{equation}
\label{bardelta}
	{\begin{aligned}
			\left\|\sum_{j=t_{i}}^{i} \rho_{j}\right\|^{2} & \lesssim \left(\left\|S_{t_{i-1}}^{i} \rho_{t_{i}-1}\right\|^{2}+\left(\sum_{p=t_{i}}^{i}\left\|\mathbf{I}_n+S_{p}^{i}\right\| \alpha_{p-1}\left\|\zeta_{p-1}+\epsilon_{p-1}\right\|\right)^{2}\right) \\
			& \lesssim\left\|S_{t_{i-1}}^{i}\right\|^{2} \left\|\rho_{t_{i}-1}\right\|^{2}
			+
			\left(\sum_{p=t_{i}}^{i}\left\|\mathbf{I}_n+S_{p}^{i}\right\|^{2} \alpha_{p-1}^{2}\right)\left(\sum_{p=t_{i}}^{i} \left\|\zeta_{p-1}+\epsilon_{p-1}\right\|^{2}\right),
	\end{aligned}}
\end{equation}
where the first inequality follows from triangle inequality and the second inequality follows from  Cauchy-Schwartz inequality.
	According to \cite[Lemma A.2.]{zhu2021online},
	$
	\left\|S_{t_{i}-1}^{{i}}\right\|^{2} \lesssim t_{i}^{2 \beta}.
	$
	On the other hand, Lemma \ref{rho_na.s.} {implies}
	$$
	\left\|\rho_{t_{i}-1}\right\|^{2} \lesssim o(t_{i}^{-2\beta\gamma} )~~a.s..
	$$
	Following from \cite[Lemma A.2.]{zhu2021online},
	$
	\sum_{p=t_{i}}^{i}\left\|\mathbf{I}_n+S_{p}^{i}\right\|^{2} \alpha_{p-1}^{2}
	\lesssim
	l_i,
	$
and following from \cite[Theorem 3]{Duchi19}, $\left\|\varepsilon_p \right\|=0$  almost surely for sufficiently large $p$. Then,
	$$
	  \sum_{p=t_{i}}^{{i}}\left\|\zeta_{p-1}+\epsilon_{p-1}\right\|^{2}
	  \lesssim_{r} \sum_{p=t_{i}}^{{i}}   C^{2}\left\|\Delta_{p-1}\right\|^{4} \lesssim  l_{i} t_{i}^{-4\beta\delta} ~~a.s..
	$$
Subsequently,
	$$
	\left\|\sum_{j=t_{i}}^{i} \rho_{j}\right\|^{2} \lesssim t_{i}^{2 \beta-2\beta\gamma} +l_{i}^{2} t_{i}^{-4 \beta\delta}.
	$$
	{Note that $\beta\in(7/9,1),~\delta\in(\frac{3-\beta}{8\beta},1-\frac{1}{2\beta})$,  $\gamma\in(\frac{3\beta-1}{4\beta},2\delta+1-\frac{1}{\beta})$, $\left(\sum_{i=1}^{k} l_{i}\right)^{-1} \asymp\left(\sum_{m=1}^{M} n_{m}^{2}\right)^{-1}$ and  $n_{m}=a_{m+1}-a_m$, the first term on the right hand of (\ref{eq:rho})
		\begin{equation}
			\label{bardelta1}
			\left(\sum_{i=1}^{k} l_{i}\right)^{-1} \sum_{i=1}^{k} \left\|\sum_{j=t_{i}}^{i} \rho_{j}\right\|^{2} 
			 \lesssim\left(\sum_{m=1}^{M} n_{m}^{2}\right)^{-1}\left(\sum_{m=1}^{M} \sum_{i=a_{m}}^{a_{m+1}-1}\left(a_m^{2 \beta-2\beta\gamma} +l_{i}^{2} a_m^{-4 \beta\delta}\right)\right) \to 0.
		\end{equation}
	}
On the other hand, by the definition of $\bar{\rho}_{k}$,
$$
\left\| \bar{\rho}_{k}\right\|\leq k^{-2}\left(\sum_{p=1}^{k}\|\mathbf{I}_n+S^k_p\|^2\alpha_{p-1}^2\right)\left(\sum_{p=1}^{{k}}\left\|\zeta_{p-1}+\epsilon_{p-1}\right\|^{2}\right).
$$ 
From \cite[(77)]{zhu2021online}, $\left(\sum_{i=1}^{k} l_{i}\right)^{-1} \sum_{i=1}^{k} l_{i}^{2}\leq n_M$, where $n_M=k-a_M+1$,
%. When $\beta\in(5/7,1),~\delta\in(\frac{1+\beta}{8\beta},1-\frac{1}{2\beta})$
the second term on the right hand of (\ref{eq:rho})
{\begin{equation}
\label{bardelta2}
	\begin{aligned}
	\left(\sum_{i=1}^{k} l_{i}\right)^{-1} \sum_{i=1}^{k} l_{i}^{2} \left\|\bar{\rho}_{k}\right\|^{2}
	\lesssim_{r}
	k^{-4\beta\delta}n_M\to0.
	\end{aligned}
	\end{equation}}
Combining (\ref{bardelta1}) and (\ref{bardelta2}),
%	$$
%	\frac{1}{M} \sum_{k=1}^{M} n_{k}  \left\|\bar{\delta}_{n_{k}}\right\|_{2}^{2} \lesssim_{r}
%	M^{\frac{\beta-2\beta\delta}{1-\beta}}	N^{\frac{2\beta-2\beta\delta-1}{1-\beta}}
%	+
%	{M}^{-\frac{4\beta\delta}{1-\beta}}	N^{\frac{1-\beta-4\beta\delta}{1-\beta}}~~a.s.
%	$$
%	{Since $\beta\in(2/3,1),~\delta\in(0,1-1/2\beta)$, $M=K^{\frac{1-\beta}{3}}$ and $N=\frac{K^{1-\beta}}{M+1}$, for sufficiently large $K$,
%	$$
%M^{\frac{\beta-2\beta\delta}{1-\beta}}	N^{\frac{2\beta-2\beta\delta-1}{1-\beta}}
%+
%{M}^{-\frac{4\beta\delta}{1-\beta}}	N^{\frac{1-\beta-4\beta\delta}{1-\beta}}\to0~~a.s.
%	$$}
(\ref{eq:rho}) converges to zero in distribution.

%In what follows, we focus on the second term on the right hand of (\ref{dec}).
%According to Lemma \ref{N_nto0},  $M^{-1}\sum_{i=1}^{M} n_{i}\left(\bar{U}_{n_{i}}-\bar{U}_{M}\right)\left(\bar{U}_{n_{i}}-\bar{U}_{M}\right)^{T}$ converges to $P_AH^\dagger P_A\Sigma P_AH^\dagger P_A$ in distribution.
%%and\\ ${M}^{-1} \sum_{k=1}^{M} n_{k}\left(\bar{\delta}_{n_{k}}-\bar{\delta}_{M}\right)\left(\bar{\delta}_{n_{k}}-\bar{\delta}_{M}\right)^{T}$ tends to zero
%Then
%by Slutsky's Theorem, the second term on the right hand of (\ref{dec}) tends to zero in distribution. 
The proof is complete.
\end{proof}

With    Lemmas \ref{N_nto0} and \ref{delta_kto0} at hand, obtaining   the consistency of batch-means estimator in distribution is standard.
\begin{thm}
	\label{XX}
	Suppose that (i) Assumptions \ref{assu_1}, \ref{ass-batch}  and \ref{plugass} hold, (ii)
{step-size
	$\alpha_{k} =\alpha_{0} k^{-\beta} \text { with }
	\beta  \in\left(\frac{7}{9}, 1\right)$ and $ \alpha_{0}>0$,} (iii) $a_m = [Cm^\tau]_+$, where $C > 0$ and $\tau > 1/(1-\beta)$. Then,
	$$\left\|\frac{\sum_{i=1}^{k}\left(\sum_{j=t_{i}}^{i} x_{j}-l_{i} \bar{x}_{k}\right)\left(\sum_{j=t_{i}}^{i} x_{j}-l_{i} \bar{x}_{k}\right)^{T}}{\sum_{i=1}^{k} l_{i}}
	-
	P_AH^\dagger P_A\Sigma P_AH^\dagger P_A\right\|\stackrel{d}{\longrightarrow}0$$ $\text{as} ~k\to\infty$.
\end{thm}

\section{Numerical test}
In this section, we  report some preliminary numerical results on the confidence regions of the solution
for SVIP (\ref{svi}).
%In this section, we report some preliminary numerical results on the confidence regions of the  solution to  distributed stochastic optimization.
Following the asymptotic distribution given in Theorem \ref{thm:SA-c},
\begin{equation*}%\label{confidence region}
\left\{y:\left(y-\bar{x}\right)^T\Gamma^{-1}\left(y-\bar{x}\right)\le \frac{1}{k}\chi_\alpha^2(n)\right\}
\end{equation*}
defines an approximate $(1-\alpha)$ confidence region for  the  solution to SVIP, where $\Gamma:=\mathrm{P}_{A}H^{\dagger} \mathrm{P}_{A}\\ \Sigma \mathrm{P}_{A}H^{\dagger} \mathrm{P}_{A}$,   $\bar{x}\define\frac{1}{k}\sum_{t=0}^{k}x_{t}$, $\chi_{\alpha}^{2}(n)$ is defined to be the number that satisfies $P\left(U>\chi_{\alpha}^{2}(n)\right)=\alpha$ for a $\chi^{2}$ random variable $U$ with $n$ degrees of freedom.
Similarly, the approximate  $(1-\alpha)$ confidence region for {the asymptotic} distribution given in Theorem \ref{thm:non-ave} is
\begin{equation*}%\label{confidence region}
\left\{y:\left(y-{x}_k\right)^T\tilde{\Sigma}^{-1}\left(y-{x}_k\right)\le \alpha_{k}\chi_\alpha^2(n)\right\},
\end{equation*}
where $\tilde{\Sigma}$ is defined in (\ref{Sigma1}).

Compared with confidence regions, the
individual confidence intervals of the solution induce a measure of the uncertainty
in each individual component of an estimated solution. Then  it is  able to
assess the uncertainty in an individual component, which thereby allows us
to focus on parameters of specific component of our interest.  Under Theorem \ref{thm:SA-c}, the approximate  $(1-\alpha)$ confidence interval for $j$-th component of  solution is
\begin{equation*}
\left\{y:\bar{x}(j)-z_{\alpha/2}\sqrt{\frac{\Gamma(j,j)}{k}}\le y\le \bar{x}(j)+z_{\alpha/2}\sqrt{\frac{\Gamma(j,j)}{k}} \right\},
\end{equation*}
where $\bar{x}(j)$ and $\Gamma(j,j)$ are the $j$-th and $(j,j)$-th components of  $\bar{x}$ and $\Gamma$ respectively,
$z_{\alpha/2}$ satisfies $P\left(U>z_{\alpha/2}\right)=\alpha/2$ for the standard normal random variable $U$.
Similarly, the approximate  $(1-\alpha)$ individual confidence interval for $j$-th component of solution  under {Theorem \ref{thm:non-ave}} is
\begin{equation*}
\left\{y:{x}_k(j)-z_{\alpha/2}\sqrt{{\alpha_{k}\tilde{\Sigma}(j,j)}}\le y\le {x}_k(j)+z_{\alpha/2}\sqrt{{\alpha_{k}\tilde{\Sigma}(j,j)}} \right\}.
\end{equation*}

{We report the empirical performance of the proposed methods on two examples from \cite{LuMOR13} and \cite{LuLASSO},}  where the first example is
a stochastic linear complementarity problem with  simulated data and the  second example is  a  linear regression problem with real data \cite[Prostate cancer]{Hastie09}.

\subsection{Stochastic linear complementarity problem}
{We first consider} a stochastic linear complementarity problem \cite{LuMOR13}:
\begin{equation}
\label{SVIP}
0 \leq \mathbb{E}[F(x, \xi)] \perp x \geq 0,
\end{equation}
where
\begin{equation*}
F(x, \xi)=\left[\begin{array}{ll}
\xi_{1} & \xi_{2} \\
\xi_{3} & \xi_{4}
\end{array}\right]\left[\begin{array}{l}
x_{1} \\
x_{2}
\end{array}\right]-\left[\begin{array}{l}
15 \\
30
\end{array}\right]+\left[\begin{array}{l}
\xi_{5} \\
\xi_{6}
\end{array}\right]
\end{equation*}
and $\xi=\{\xi_1,\cdots,\xi_{6}\}$ follows uniform distribution over the box
$$
\left\{\xi \in \R^{6} \mid(0,0,0,0,-1,-1) \leq \xi \leq(2,1,2,4,1,1)\right\}.
$$
{Obviously},
  the unique  true solution $x^* =(10,10)^T$ and
the true covariance matrices in Theorem \ref{xn-x*} and Theorem \ref{thm:non-ave} are
 $$\left[\begin{array}{ll}
111.78 & -55.78\\
-111.56 &  83.56
\end{array}\right],  \qquad \left[\begin{array}{ll}
30.31 & -18.61\\
-18.61 &  51.06
\end{array}\right]$$
 respectively.
 
\begin{figure}[htbp]
	\centering
	\subfigure[Asymptotic normality: average]{
		\includegraphics[width=3in]{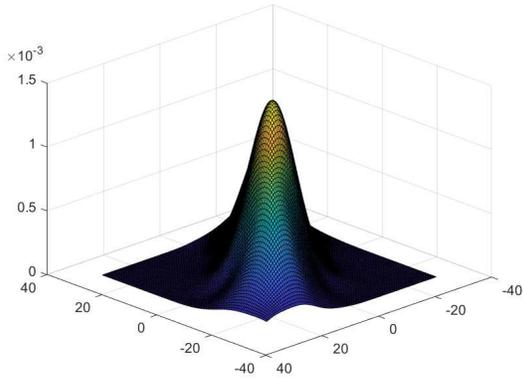}
	}
	\subfigure[Asymptotic normality: last iterate]{
		\includegraphics[width=3in]{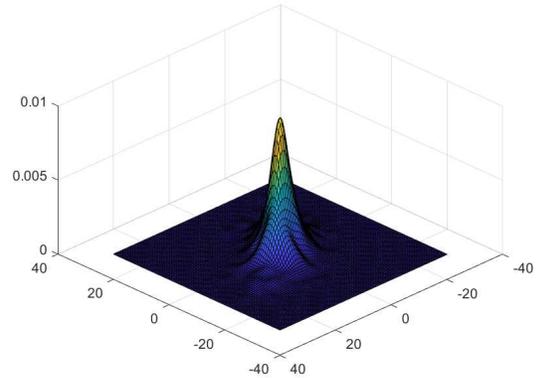}
	}
	\caption{Asymptotic normality of SDA}
		\label{f1}
\end{figure}

In implementing  of  Algorithm \ref{AL1}, the step-size {$\alpha_{k}=0.5 *k^{-0.81}$} and the initial point $x_{0}=(0,0)^T$. We first test the asymptotic normality of iterates of SDA   in Theorems \ref{thm:SA-c} and \ref{thm:non-ave}.
 We do 1000 Monte-Carlo simulations
of running SDA  1000 {iterates} and record the  estimated density in Figure \ref{f1}. Figure  \ref{f1} (a) and Figure  \ref{f1} (b) depict the estimated {densities} of  the average of
iterates SDA and the last iterate of SDA respectively.
Figure \ref{f1} seems to confirm Theorems \ref{thm:SA-c} and \ref{thm:non-ave} since {we can see that the estimated density  is close
to the density of a normal distribution} and  is also confirmed by a Kolmogorov-Smirnov test.

Next, we record the {$90\%$ confidence regions} with number of {iterates}  $k=1000$, $2000$ and $5000$   respectively.  For the stability, we do 50 Monte-Carlo simulations and report the results with the average covariance matrix and the average of {iterates}. {In batch-means method, the sequence $\{a_m\}$ is chosen in the form $a_m = \left[ Cm^\frac{2}{1-\beta}\right] _+$ with $C=1$.}  Figure \ref{f2}  depicts the asymptotic confidence regions  of the solution to {complementarity problem (\ref{SVIP})}, where the red circle ellipse, blue dashed ellipse, green dot ellipse and black solid ellipse denote the confidence regions  {for number of iterates} $1000$, $2000$, $5000$ and the true one respectively.
 As we can observe from Figure \ref{f2} (a), the asymptotic confidence region based on plug-in method at $k=5000$ {almost coincides} with  true confidence region, which indicates the consistency of plug-in method in Theorem \ref{T3.1}. Compared with  Figure \ref{f2} (a),    the asymptotic confidence region based on batch-means methods is reported in  Figure \ref{f2} (b),  where the asymptotic confidence region at $k=5000$ {is small than} the true one. {The underlying reason {may be} that the plug-in method employs more information {such as gradient  of   functions} and batch-means method uses iterates of SDA only. On the other hand, the batch-means estimator tends to underestimate the variance due to the correlation between batches.}  Figure \ref{f2} (c)  verifies  the consistency of  plug-in method in
building the  asymptotic confidence regions based on the last iterate of SDA.  %Moreover, compared \ref{f2} (a) with \ref{f2} (c),  we conclude that the average of iterations  outperforms the last iterate

\begin{figure}[http]
	\centering
	\begin{minipage}[t]{0.31\textwidth}
		\centering
		\subfigure[Plug-in]{
			\includegraphics[height=2.2in,width=2.2in]{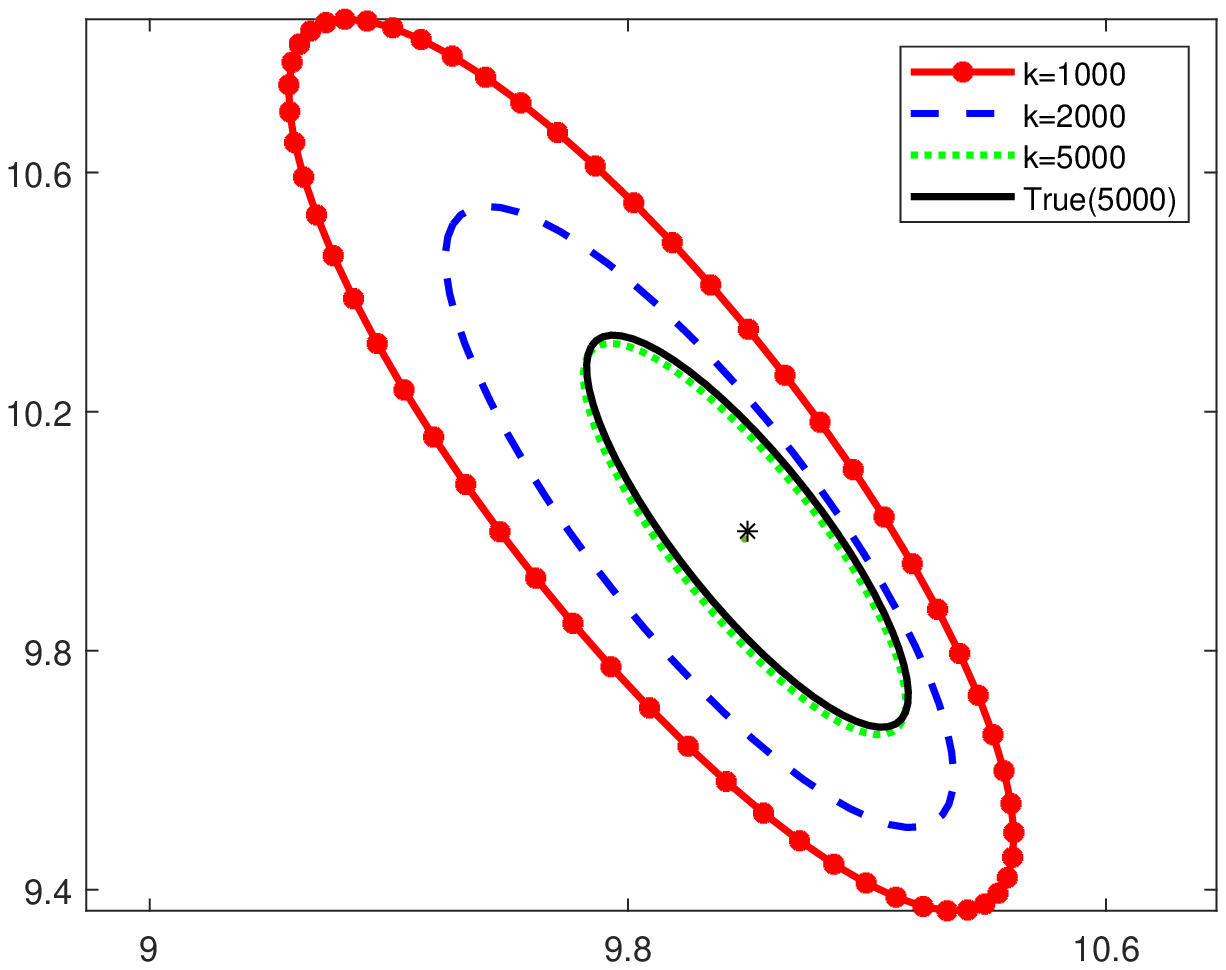}
		}
	\end{minipage}
	\begin{minipage}[t]{0.31\textwidth}
		\centering
		\subfigure[Batch-means]{
			\includegraphics[height=2.2in,width=2.2in]{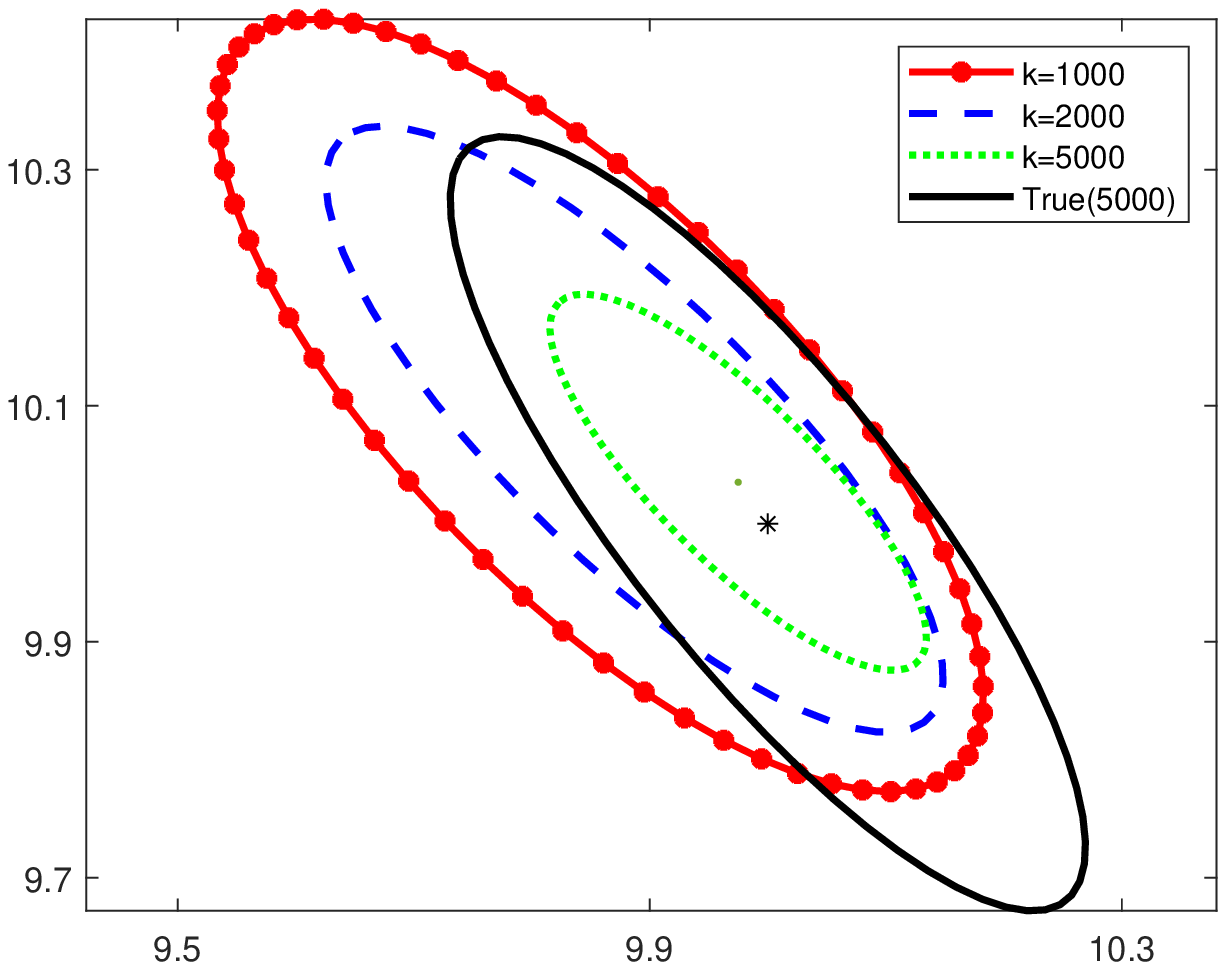}
		}
	\end{minipage}
	\begin{minipage}[t]{0.31\textwidth}
		\centering
		\subfigure[Plug-in (Non-ergodic)]{
			\includegraphics[height=2.2in,width=2.2in]{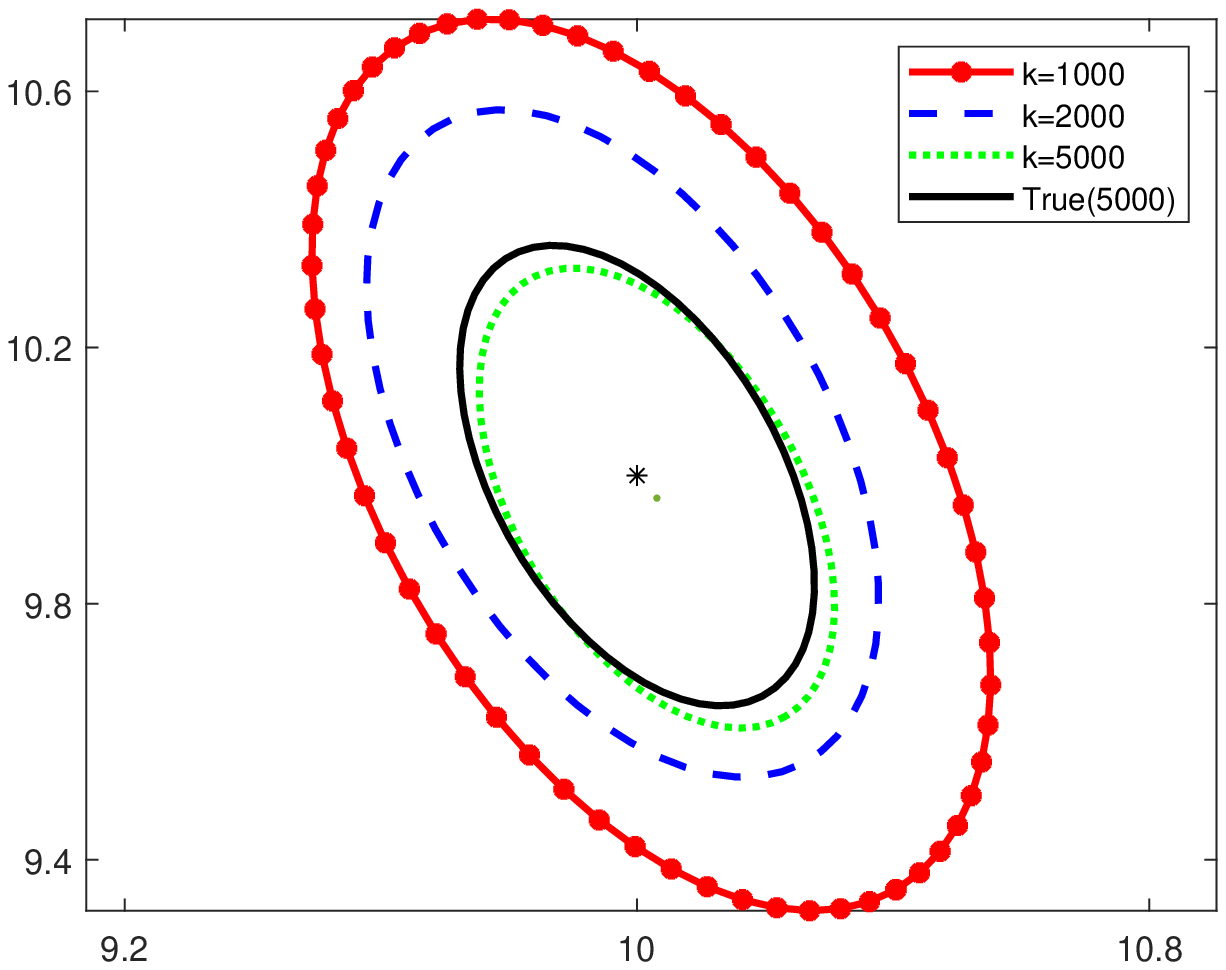}
		}
	\end{minipage}
	\caption{Asymptotic confidence regions for (\ref{SVIP})}
	\label{f2}
\end{figure}

We record the   diagonal elements of covariance matrices {for number of iterates $1000$, $2000$, $5000$ and the true one respectively} in
Table \ref{t1}-\ref{t2}, which characterize the individual confidence intervals of the   solution to complementarity problem (\ref{SVIP}).
{Similar to Figure \ref{f2}, we can conclude that  the plug-in estimators are consistent.}

\begin{table}[h]
	\begin{center}
		\caption{Diagonal elements of $\Gamma$}
		\label{t1}
		\begin{tabular}{l|lll|lll|l}
			\hline
			& \multicolumn{3}{c|}{Plug-in} & \multicolumn{3}{c|}{Batch-means} & \multicolumn{1}{c}{\multirow{2}{*}{TRUE}} \\ \cline{2-7}
			Iterations & 1000 & 2000 & 5000 & 1000 & 2000 & 5000 & \multicolumn{1}{c}{} \\ \hline
			$\Gamma(1,1)$& 114.07 & 111.65 & 112.05
			 & 23.25 & 29.68 & 27.61 & 111.78 \\
			$\Gamma(2,2)$ & 86.52 & 83.99 & 83.15
			& 22.87 & 28.61 & 27.57 & 83.56 \\
			\hline
		\end{tabular}
	\end{center}
\end{table}

\begin{table}[h]
	\begin{center}

		\caption{Diagonal elements of $\tilde{\Sigma}$}
		\label{t2}
		\begin{tabular}{l|lll|l}
			\hline
			Iterations & 1000 & 2000 & 5000 & \multicolumn{1}{c}{TRUE} \\ \hline
		$\tilde{\Sigma}(1,1)$ & 30.63 & 30.36 & 30.42 & 30.31 \\
		$\tilde{\Sigma}(2,2)$ & 52.84 & 51.53 & 50.94 & 51.06 \\
 \hline
		\end{tabular}
	\end{center}
\end{table}

{
	We report the coverage probability of $90\%$ confidence regions in Table \ref{coverrate}.
	We estimate the coverage probability
	 by 1000 replications. From Table \ref{coverrate}, we can observe that the coverage probabilities of the plug-in methods are getting closer to
	the nominal level $90\%$ when the number of iterates $k$ grows larger. However, the coverage probabilities of the batch-means method is only $14\%$. 
	The underestimation problem of the batch-means method is because it neglects the correlation between batches. 
	One possible way to handle this problem is to do Monte-Carlo simulation as in Figure \ref{f2}}.
% \footnote{We do 10 Monte-Carlo simulations with the number of iterates $1000$, where the estimators of iterates and covariance matrix are averaged. Then,  the coverage probability is around $82\%$.}

\begin{table}[h]
	\begin{center}
		\caption{The coverage probability of $90\%$ confidence regions}
		\label{coverrate}
		\begin{tabular}{l|lll}
		\hline
		iterations & 1000 & 2000 & 5000 \\ \hline
		Plug-in & 82 & 84 & 88 \\
		Batch-means & 14 & 14 & 20 \\
		Plug-in (Non-ergodic) & 83 & 83 & 86 \\ \hline
	\end{tabular}
	\end{center}
\end{table}

\subsection{Lasso}
Least absolute shrinkage and selection operator (Lasso) is  a regression analysis method that performs both variable selection and regularization in order to enhance the prediction accuracy and interpretability of the resulting statistical model.
{We consider} lasso on the prostate cancer example \cite{LuLASSO},
\begin{equation}
\label{lasso2}
\min _{\left(\beta_{0}, \beta, t\right) \in \C}\left(\mathbb{E}\left[Y-\beta_{0}-\sum_{j=1}^{8} \beta_{j} X_{j}\right]^{2}+\lambda \sum_{j=1}^{8} t_{j}\right),
\end{equation}
where $X \in \R^{8}$ is the random input vector and $Y \in \R$ is the response variable.
The feasible set $\C$ of problem (\ref{lasso2}) is given by
$$
\C=\left\{\left(\beta_{0}, \beta, t\right) \in \R \times \R^{8} \times \R^{8} \mid t_{j}-\beta_{j} \geqslant 0, t_{j}+\beta_{j} \geqslant 0, j=1, \ldots, 8\right\}.
$$
%Stamey et al. \cite{STAMEY19891084} examined the correlation between the level of prostate-specific antigen and a number of clinical measures in men who were about to receive a radical prostatectomy.
Similar to \cite{LuLASSO},
we first standardize the predictors to have unit variance and split observations into two parts. One part consists of 67 observations, which are the training set in \cite{Hastie09}. We use only these 67 observations in our computation.
In implementing  of  Algorithm \ref{AL1}, we use the same setting of the  step-size and  initial point in the former example, that is,   {$\alpha_{k}=0.5k^{-0.81}$} and $x_{0}=0_{17}$. Moreover, the maximum number of {iterates} is $k=3000$.

\begin{table}[h]
	\centering
	%\fontsize{6.5}{8}
	\selectfont
	\caption{95\% confidence intervals  for {$\lambda=0.45$}}
	\label{t3}
	\begin{tabular}{c|ccc|cc}
		\hline
		          & Ave-Est  & PI CI           & BM CI           & Last-Est         &Non-PI CI \\ \hline
		$\beta_1$ & 0.57 & {[}0.55,0.60{]} & {[}0.56,0.59{]} &0.54             &{[}0.44,0.65{]} \\
		$\beta_2$ & 0.17 & {[}0.13,0.20{]} & {[}0.14,0.20{]} & 0.18            & {[}0.09,0.27{]} \\
		$\beta_3$ & 0    & {[}0,0.04{]}    & {[}0,0.01{]}       & 0               & {[}0,0.10{]} \\
		$\beta_4$ & 0.01 & {[}0.01,0.01{]} & {[}0,0.03{]} & 0               & {[}0,0{]} \\
		$\beta_5$ & 0.09 & {[}0.09,0.09{]} & {[}0.07,0.10{]} & 0.08            & {[}0.08,0.08{]} \\
		$\beta_6$ & 0    & {[}0,0{]}       & {[}0,0{]}       & 0               & {[}0,0{]} \\
		$\beta_7$ & 0    & {[}0,0{]}       & {[}0,0{]}       & 0               & {[}0,0{]} \\
		$\beta_8$ & 0    & {[}0,0{]}       & {[}0,0{]}       & 0               & {[}0,0{]} \\ \hline
	\end{tabular}
\end{table}

\begin{table}[h]
	\centering
	%\fontsize{6.5}{8}
	\selectfont
	\caption{95\% confidence intervals   for {$\lambda=1.49$}}
	\label{t4}
	\begin{tabular}{c|ccc|cc}
		\hline
		          & Ave-Est  & PI CI           & BM CI           & Last-Est         &Non-PI CI \\ \hline
		$\beta_1 $& 0.21 & {[}0.14,0.28{]} & {[}0.18,0.24{]} & 0.17            & {[}0.03,0.31{]} \\
		$\beta_2 $& 0    & {[}0,0{]}       & {[}0,0{]}       & 0               & {[}0,0{]} \\
		$\beta_3 $& 0    & {[}0,0{]}       & {[}0,0{]}       & 0               & {[}0,0{]} \\
		$\beta_4 $& 0    & {[}0,0{]}       & {[}0,0.01{]}       & 0               & {[}0,0{]} \\
		$\beta_5 $& 0    & {[}0,0{]}       & {[}0,0.02{]}       & 0               & {[}0,0{]} \\
		$\beta_6 $& 0    & {[}0,0{]}       & {[}0,0.01{]}       & 0               & {[}0,0{]} \\
		$\beta_7 $& 0    & {[}0,0{]}       & {[}0,0.01{]}       & 0               & {[}0,0{]} \\
		$\beta_8$ & 0    & {[}0,0{]}       & {[}0,0.01{]}       & 0               & {[}0,0{]} \\ \hline
	\end{tabular}
\end{table}

Tables \ref{t3}-\ref{t4} record   the $95\%$ individual confidence intervals   for lasso  parameters  with  penalty terms $\lambda=0.45$  and $1.49$ respectively. We only report the confidence regions of $\beta_1, \cdots, \beta_8$ as
$\beta_0$ is the intercept.
 Similar to \cite{LuLASSO}, we can conclude the importance of predictors  in predicting the
response and the impact  of  penalty term $\lambda$ in sparseness of predictors to problem (\ref{lasso2}). Specifically, for $\lambda=0.45$, the individual confidence intervals of $\beta_{1}$ and $\beta_{2}$ do not contain zero  and  the variances  related to $\beta_{4}$ and $\beta_{5}$ are zero.
Moreover, the individual confidence intervals of all other variables   include zero in them.  As $\beta_4=0.01$ and $\beta_5=0.09$ are close to zero,  we may claim that  the first two predictors are the most useful ones in predicting the response.  On the other hand,  for $\lambda=1.49$, only  the individual confidence interval of $\beta_{1}$ does not contain
zero, which indicates that  the first predictor is more important than the second one.
{We} can also observe from Tables \ref{t3}-\ref{t4} that lasso shrinks the regression coefficients by imposing a penalty parameter  $\lambda$
on their size.

\bibliographystyle{ieeetr}
\bibliography{myref}

\section{Appendix}
\begin{lema}\emph{\cite[Lemma 3.1.1]{Chen06}}\label{lem:rate}
	Suppose $n\times n$-dimension matrix $F_k\rightarrow F$, $F$ is a stable matrix, that is, every eigenvalue of $F$ has strictly negative real part. If step-size $\alpha_{k}$ satisfies {$\alpha_k>0,\alpha_{k}\rightarrow0$ as $k\rightarrow\infty$, $\sum_{k=1}^\infty\alpha_{k}=\infty$}
	and $n$-dimension vectors $\{e_k\},\{\upsilon_k\}$ satisfy the following conditions
	\begin{equation*}\label{rate condition}
	\sum_{k=1}^\infty \alpha_{k}e_k<\infty,~\upsilon_k\rightarrow 0,
	\end{equation*}
	then $\{y_k\}$ defined by the following recursion with arbitrary initial value $y_0$ tends to zero:
	\begin{equation}\label{linear reccursion}
	y_{k+1}=y_k+\alpha_{k}F_ky_k+\alpha_{k}\left(e_k+\upsilon_k\right).
	\end{equation}
\end{lema}

\begin{lema}\label{lem:asym norm} \emph{\cite[Theorem 3.3.1]{Chen06}}
	Let $\{y_k\}$ be given by the following recursion with an arbitrarily given initial value:
	\begin{equation}\label{linear reccursion_1}
	y_{k+1}=y_k+\alpha_{k}F_ky_k+\alpha_{k}\left(e_k+\upsilon_k\right).
	\end{equation}
	 Assume the following {conditions hold}:
	\begin{itemize}
		\item [\rm{(i)}]$\alpha_k>0,\alpha_{k}\rightarrow0$ as $k\rightarrow\infty$, $\sum_{k=1}^\infty\alpha_{k}=\infty$ and
		\begin{equation*}
		\alpha_{k+1}^{-1}-\alpha_{k}^{-1}\rightarrow a\ge 0~\text{as}~k\rightarrow\infty;
		\end{equation*}
		\item [\rm{(ii)}] $F_k\rightarrow F$ and $F+\dfrac{a}{2}\mathbf{I}_n$ is stable;
		\item [\rm{(iii)}]
		\begin{equation*}
		{\upsilon_k}=o(\sqrt{\alpha_{k}}),   \quad e_k=\sum_{t=0}^\infty C_ts_{k-t},s_t=0~\text{for}~t<0,
		\end{equation*}
		where $C_t$ are $n\times n$ constant matrices with $\sum_{t=0}^\infty \|C_t\|<\infty$ and $\{s_k,\mathcal{F}_k\}$ is a martingale difference sequence of $n$-dimension satisfying the following conditions
		\begin{equation}\label{c1}
		 \mathbb{E}\left[s_k|\mathcal{F}_{k-1}\right]=0,~\sup_k\mathbb{E}\left[\|s_k\|^2\big|\mathcal{F}_{k-1}\right]\le\sigma~\text{with}~\sigma~\text{being a constant,}
		\end{equation}
		\begin{equation}\label{c2}
		 \lim_{k\rightarrow\infty}\mathbb{E}\left[s_ks_k^T\big|\mathcal{F}_{k-1}\right]=\lim_{k\rightarrow\infty}\mathbb{E}\left[s_ks_k^T\right]\define S_0
		\end{equation}
		and
		\begin{equation}\label{c3}
		\lim_{N\rightarrow\infty}\sup_{k}\mathbb{E}\left[\|s_k\|^21_{\{\|s_k\|>N\}}\right]=0.
		\end{equation}
		Then $\dfrac{y_k}{\sqrt{\alpha_{k}}}$ is asymptotically normal:
		{\begin{equation*}
		\dfrac{y_k}{\sqrt{\alpha_{k}}}\xrightarrow{d}\mathcal{N}(0,S),
		\end{equation*}}
		where
		\begin{equation*}
		S=\int_{0}^{\infty}e^{(F+(a/2)\mathbf{I}_n)t}\sum_{k=0}^\infty C_kS_0\sum_{k=0}^\infty C_k^Te^{(F^T+(a/2)\mathbf{I}_n)t}dt.
		\end{equation*}
	\end{itemize}
\end{lema}
\end{document}